\documentclass[reqno, 12pt]{amsart}
\pdfoutput=1
\makeatletter
\let\origsection=\section \def\section{\@ifstar{\origsection*}{\mysection}}
\def\mysection{\@startsection{section}{1}\z@{.7\linespacing\@plus\linespacing}{.5\linespacing}{\normalfont\scshape\centering\S}}
\makeatother

\usepackage{amsmath,amssymb,amsthm}
\usepackage{mathrsfs}
\usepackage{mathabx}\changenotsign
\usepackage{dsfont}
\usepackage{bbm}

\usepackage{xcolor}
\usepackage[backref=section]{hyperref}
\usepackage[ocgcolorlinks]{ocgx2}
\hypersetup{
	colorlinks=true,
	linkcolor={red!60!black},
	citecolor={green!60!black},
	urlcolor={blue!60!black},
}

\usepackage[open,openlevel=2,atend]{bookmark}

\usepackage[abbrev,msc-links,backrefs]{amsrefs}
\usepackage{doi}

\renewcommand{\PrintDOI}[1]{\doi{#1}}

\usepackage[T1]{fontenc}
\usepackage{lmodern}
\usepackage[babel]{microtype}
\usepackage[english]{babel}

\usepackage{geometry}
\numberwithin{equation}{section}
\numberwithin{figure}{section}

\usepackage{enumitem}

\def\alabel{\upshape({\itshape \alph*\,})}

\def\nlabel{\upshape({\itshape \arabic*\,})}

\let\polishlcross=\l
\def\l{\ifmmode\ell\else\polishlcross\fi}

\let\emptyset=\varnothing
\let\setminus=\smallsetminus

\makeatletter
\def\moverlay{\mathpalette\mov@rlay}
\def\mov@rlay#1#2{\leavevmode\vtop{   \baselineskip\z@skip \lineskiplimit-\maxdimen
		\ialign{\hfil$\m@th#1##$\hfil\cr#2\crcr}}}
\newcommand{\charfusion}[3][\mathord]{
	#1{\ifx#1\mathop\vphantom{#2}\fi
		\mathpalette\mov@rlay{#2\cr#3}
	}
	\ifx#1\mathop\expandafter\displaylimits\fi}
\makeatother

\newcommand{\dcup}{\charfusion[\mathbin]{\cup}{\cdot}}
\newcommand{\bigdcup}{\charfusion[\mathop]{\bigcup}{\cdot}}

\DeclareFontFamily{U}  {MnSymbolC}{}
\DeclareSymbolFont{MnSyC}         {U}  {MnSymbolC}{m}{n}
\DeclareFontShape{U}{MnSymbolC}{m}{n}{
	<-6>  MnSymbolC5
	<6-7>  MnSymbolC6
	<7-8>  MnSymbolC7
	<8-9>  MnSymbolC8
	<9-10> MnSymbolC9
	<10-12> MnSymbolC10
	<12->   MnSymbolC12}{}
\DeclareMathSymbol{\powerset}{\mathord}{MnSyC}{180}
\DeclareMathSymbol{\YY}{\mathord}{MnSyC}{42}

\usepackage{tikz}
\usetikzlibrary{calc,decorations.pathmorphing}
\usetikzlibrary{arrows,decorations.pathreplacing}
\pgfdeclarelayer{background}
\pgfdeclarelayer{foreground}
\pgfdeclarelayer{front}
\pgfsetlayers{background,main,foreground,front}
\definecolor{uuuuuu}{rgb}{0.27,0.27,0.27}
\definecolor{sqsqsq}{rgb}{0.1255,0.1255,0.1255}

\usepackage{multicol}
\usepackage{subcaption}
\let\epsilon=\varepsilon
\let\eps=\epsilon
\let\phi=\varphi
\let\rho=\varrho
\let\theta=\vartheta

\def\NN{{\mathds N}}

\def\RR{{\mathds R}}

\def\ex{{\mathrm{ex}}}

\newcommand{\cF}{\mathcal{F}}
\newcommand{\cG}{\mathcal{G}}
\newcommand{\cH}{\mathcal{H}}

\newcommand{\cK}{\mathcal{K}}

\newcommand{\cM}{\mathcal{M}}

\newcommand{\cS}{\mathcal{S}}
\newcommand{\cT}{\mathcal{T}}

\newcommand{\gK}{\mathfrak{K}}
\newcommand{\gS}{\mathfrak{S}}
\newcommand{\gT}{\mathfrak{T}}
\newcommand{\gH}{\mathfrak{H}}
\newcommand{\gh}{\mathfrak{h}}

\newtheoremstyle{note}  {4pt}  {4pt}  {\sl}  {}  {\bfseries}  {.}  {.5em}          {}
\newtheoremstyle{introthms}  {3pt}  {3pt}  {\itshape}  {}  {\bfseries}  {.}  {.5em}          {\thmnote{#3}}
\newtheoremstyle{remark}  {2pt}  {2pt}  {\rm}  {}  {\bfseries}  {.}  {.3em}          {}

\theoremstyle{plain}
\newtheorem{theorem}{Theorem}[section]
\newtheorem{lemma}[theorem]{Lemma}

\newtheorem{corollary}[theorem]{Corollary}

\newtheorem{cor}[theorem]{Corollary}
\newtheorem{fact}[theorem]{Fact}
\newtheorem{claim}[theorem]{Claim}

\theoremstyle{note}
\newtheorem{dfn}[theorem]{Definition}
\newtheorem{definition}[theorem]{Definition}

\theoremstyle{remark}

\usepackage{lineno}
\newcommand*\patchAmsMathEnvironmentForLineno[1]{
	\expandafter\let\csname old#1\expandafter\endcsname\csname #1\endcsname
	\expandafter\let\csname oldend#1\expandafter\endcsname\csname end#1\endcsname
	\renewenvironment{#1}
	{\linenomath\csname old#1\endcsname}
	{\csname oldend#1\endcsname\endlinenomath}}
\newcommand*\patchBothAmsMathEnvironmentsForLineno[1]{
	\patchAmsMathEnvironmentForLineno{#1}
	\patchAmsMathEnvironmentForLineno{#1*}}
\AtBeginDocument{
	\patchBothAmsMathEnvironmentsForLineno{equation}
	\patchBothAmsMathEnvironmentsForLineno{align}
	\patchBothAmsMathEnvironmentsForLineno{flalign}
	\patchBothAmsMathEnvironmentsForLineno{alignat}
	\patchBothAmsMathEnvironmentsForLineno{gather}
	\patchBothAmsMathEnvironmentsForLineno{multline}
}

\usepackage{scalerel}

\makeatletter
\newcommand{\overrighharpoonup}[1]{\ThisStyle{%
		\vbox {\m@th\ialign{##\crcr
				\rightharpoonupfill \crcr
				\noalign{\kern-\p@\nointerlineskip}
				$\hfil\SavedStyle#1\hfil$\crcr}}}}

\def\rightharpoonupfill{%
	$\SavedStyle\m@th\mkern+0.8mu\cleaders\hbox{$\shortbar\mkern-4mu$}\hfill\rightharpoonuptip\mkern+0.8mu$}

\def\rightharpoonuptip{%
	\raisebox{\z@}[2pt][1pt]{\scalebox{0.55}{$\SavedStyle\rightharpoonup$}}}

\def\shortbar{%
	\smash{\scalebox{0.55}{$\SavedStyle\relbar$}}}
\makeatother

\let\lra=\longrightarrow

\makeatletter
\newsavebox\myboxA
\newsavebox\myboxB
\newlength\mylenA

\newcommand*\xoverline[2][0.75]{%
	\sbox{\myboxA}{$\m@th#2$}%
	\setbox\myboxB\null
	\ht\myboxB=\ht\myboxA%
	\dp\myboxB=\dp\myboxA%
	\wd\myboxB=#1\wd\myboxA
	\sbox\myboxB{$\m@th\overline{\copy\myboxB}$}
	\setlength\mylenA{\the\wd\myboxA}
	\addtolength\mylenA{-\the\wd\myboxB}%
	\ifdim\wd\myboxB<\wd\myboxA%
	\rlap{\hskip 0.5\mylenA\usebox\myboxB}{\usebox\myboxA}%
	\else
	\hskip -0.5\mylenA\rlap{\usebox\myboxA}{\hskip 0.5\mylenA\usebox\myboxB}%
	\fi}
\makeatother

\DeclareSymbolFont{symbolsC}{U}{txsyc}{m}{n}
\SetSymbolFont{symbolsC}{bold}{U}{txsyc}{bx}{n}
\DeclareFontSubstitution{U}{txsyc}{m}{n}
\DeclareMathSymbol{\strictif}{\mathrel}{symbolsC}{74}

\DeclareSymbolFont{stmry}{U}{stmry}{m}{n}
\DeclareMathSymbol\arrownot\mathrel{stmry}{"58}
\DeclareMathSymbol\Arrownot\mathrel{stmry}{"59}

\let\sm=\smallsetminus

\begin{document}
\title[A unified approach to hypergraph stability]{A unified approach to hypergraph stability}

\author{Xizhi Liu}
\address{Department of Mathematics, Statistics, and Computer Science, University of Illinois,
Chicago, IL 60607 USA}
\email{xliu246@uic.edu}
\email{mubayi@uic.edu}
\thanks{The first and second author's research is partially supported by NSF awards 
DMS-1763317 and DMS-1952767.}

\author{Dhruv Mubayi}
%

\author{Christian Reiher}
\address{Fachbereich Mathematik, Universit\"at Hamburg, Hamburg, Germany}
\email{Christian.Reiher@uni-hamburg.de}

\subjclass[2010]{}

\begin{abstract}
We present a method which provides a unified framework for  most stability theorems that 
have been proved in graph and hypergraph theory.
Our main result reduces stability for a large class of hypergraph problems to the simpler 
question of checking that a hypergraph $\mathcal H$ with large minimum degree that omits the 
forbidden structures is vertex-extendable. This means that if $v$ is a vertex of $\mathcal H$ and 
${\mathcal H} -v$ is a subgraph of the extremal configuration(s), then $\mathcal H$ is also a 
subgraph of the extremal configuration(s). In many cases vertex-extendability is quite easy to 
verify.

We illustrate our approach by giving new short proofs of
hypergraph stability results of Pikhurko, Hefetz-Keevash, Brandt-Irwin-Jiang,
Bene Watts-Norin-Yepremyan and others.
Since our method always yields minimum degree stability, which is the strongest form of 
stability, in some of these cases our stability results are stronger than what was
known earlier. 
Along the way, we clarify the different notions of stability that have been previously studied.

\end{abstract}

\maketitle

\section{Introduction}\label{SEC:introduction}
\subsection{Types of stability}\label{SUBSEC:type-stability}
For $r \ge 2$ and a family $\cF$ of $r$-uniform hypergraphs (henceforth called 
$r$-graphs), an $r$-graph $\cH$ is said to be \emph{$\cF$-free} if it contains 
no member of $\cF$ as a subgraph.
For $n\in\NN$ the {\em Tur\'{a}n number} ${\rm ex}(n,\cF)$ of $\cF$ is
the maximum number of edges in an $\cF$-free {$r$-graph} on $n$ vertices.
The {\em Tur\'{a}n density} $\pi(\cF)$
of $\cF$ is defined as $\pi(\cF) = \lim_{n \to \infty} {\rm ex}(n,\cF) / \binom{n}{r}$,
and $\cF$ is called {\em nondegenerate} if $\pi(\cF)>0$.
The study of ${\rm ex}(n,\cF)$ is perhaps the central topic in extremal graph and 
hypergraph theory.

Much is known about ${\rm ex}(n,\cF)$ when $r=2$ and one of the most famous results in this 
regard is Tur\'{a}n's theorem~\cite{TU41},
which states that for $n\ge \ell \ge 2$ there is a unique $K_{\ell+1}$-free graph 
with $n$ vertices and ${\rm ex}(n,K_{\ell+1})$ edges, namely the balanced complete $\ell$-partite graph~$T(n,\ell)$.

Tur\'{a}n's theorem was extended further by Erd\H{o}s and Stone~\cite{ES46} in the following way.
Given a family of graphs $\cF$ the {\em chromatic number} $\chi(\cF)$ of $\cF$ is defined as
\begin{align}
\chi(\cF) = \min\{\chi(F)\colon F\in \cF\}, \notag
\end{align}
where $\chi(F)$ is the chromatic number of the graph $F$.
The result of Erd\H{o}s and Stone implies $\pi(\cF) = 1-\frac{1}{\chi(\cF)-1}$ for every 
family $\cF$ of graphs; the connection to the chromatic number was first stated 
explicitly by Erd\H{o}s and Simonovits~\cite{ES66}.

Extending Tur\'{a}n's theorem to hypergraphs (i.e. $r\ge 3$) is a major problem.
For $\ell > r \ge 3$, let $K_{\ell}^{r}$ be the complete $r$-graph on $\ell$ vertices.
The problem of determining $\pi(K_{\ell}^{r})$ was raised by Tur\'{a}n~\cite{TU41} and is 
still wide open.
Erd\H{o}s offered $\$ 500$ for the determination of any $\pi(K_{\ell}^{r})$ with $\ell > r \ge 3$
and $\$ 1000$ for the determination of all $\pi(K_{\ell}^{r})$ with $\ell > r \ge 3$.

Many families $\cF$ have the property that there is a unique $\cF$-free hypergraph
$\cG$ on $n$ vertices achieving ${\rm ex}(n,\cF)$,
and moreover, every $\cF$-free hypergraph $\cH$ of size close to
${\rm ex}(n,\cF)$ can be transformed to $\cG$
by deleting and adding very few edges.
Such a property is called {\em stability} of $\cF$.
The first stability theorem was proved independently by Erd\H{o}s and Simonovits~\cite{SI68}.

\begin{theorem}[Erd\H{o}s-Simonovits]\label{THM:Erdos-Simonovits-stability}
Let $\ell \ge 2$ and let $\cF$ be a family of graphs with $\chi(\cF) = \ell+1$.
Then for every $\delta > 0$ there exist $\epsilon > 0$ and $N_0\in \NN$
such that every $\cF$-free graph on $n\ge N_0$ vertices with at 
least $(1-\epsilon) {\rm ex}(n,\cF)$ edges can be transformed to the Tur\'an graph 
$T(n,\ell)$ by deleting and adding at most $\delta n^2$ edges. \qed
\end{theorem}

The stability phenomenon has been used to determine ${\rm ex}(n,\cF)$ exactly in many cases.
It was first used by Simonovits in~\cite{SI68} to determine ${\rm ex}(n,F)$ exactly for all 
edge-critical graphs $F$ and large $n$,
and then by several authors (e.g. see \cites{FS05,KS05a,KS05b,MP07,PI13,BIJ17,NY18}) to prove 
exact results for hypergraphs.
In this article, stability will always mean stability relative to some intended class $\gH$ of 
'almost extremal' $\cF$-free graphs and we distinguish the following types of stability that 
have been studied in the literature. 

\begin{dfn}\label{d:1715}
	Let $\cF$ be a nondegenerate family of $r$-graphs, where $r\ge 2$, and let $\gH$ be a class 
	of $\cF$-free $r$-graphs.
	\begin{enumerate}[label=\alabel]
		\item\label{it:stab-e} If for every $\delta>0$ there exist $\epsilon>0$ and $N_0\in \NN$
			such that every $\cF$-free $r$-graph~$\cH$ on $n\ge N_0$ vertices 
			with $|\cH|\ge (\pi(\cF)/r!-\eps)n^r$ becomes a 
			subgraph of some member of~$\gH$ after removing at most $\delta |\cH|$ edges, then 
			$\cF$ is said to be \emph{edge-stable} with respect to $\gH$.
 		\item\label{it:stab-v} If for every $\delta>0$ there exist $\epsilon>0$ and $N_0\in \NN$
			such that every $\cF$-free $r$-graph~$\cH$ on $n\ge N_0$ vertices 
			with $|\cH|\ge (\pi(\cF)/r!-\eps)n^r$ becomes a 
			subgraph of some member of~$\gH$ after removing at most $\delta |V(\cH)|$ vertices, 
			then $\cF$ is said to be \emph{vertex-stable} with respect to $\gH$.
		\item\label{it:stab-d} If there exist $\eps>0$ and $N_0$ such that every $\cF$-free 
			$r$-graph $\cH$ on $n\ge N_0$ vertices 
			with $\delta(\cH)\ge \bigl(\pi(\cF)/(r-1)!-\epsilon\bigr)n^{r-1}$ 
			is a subgraph of some member of $\gH$ we say that $\cF$ is \emph{degree-stable}
			with respect to~$\gH$.
	\end{enumerate}
\end{dfn} 

As a trivial example, every nondegenerate family $\cF$ is stable in all three senses with 
respect to the class $\mathfrak{Forb}(\cF)$ of all $\cF$-free $r$-graphs. 
More interestingly, Theorem~\ref{THM:Erdos-Simonovits-stability} tells us that every family 
$\cF$ of graphs with $\chi(F)=\ell+1\ge 3$ is edge-stable with respect to the class 
$\gT_\ell=\{T(n, \ell)\colon n\in \NN\}$ of $\ell$-partite Tur\'an graphs. 

In general, if a family $\cF$ of $r$-graphs is degree-stable with respect to some class $\gH$,
then a standard vertex deletion argument (see e.g. 
Fact~\ref{FACT:dense-hygp-has-large-min-deg}~\ref{it:26a}) shows that $\cF$ is vertex-stable 
with respect to $\gH$ as well. Moreover, since any $\delta v(\cH)$ vertices of an $r$-graph 
$\cH$ cover at most $\delta v(\cH)^r$ edges of $\cH$, it is in all interesting examples the 
case that if $\cF$ is vertex-stable with respect to $\gH$, then it is edge-stable with respect 
to $\gH$ as well. 

The goal of this work is to provide a unified framework for the stability of certain classes 
of graph and hypergraph families.
Our main result (Theorem~\ref{THM:Psi-trick:G-extendable-implies-degree-stability})
reduces the stability of many problems to the much simpler task of checking that 
$\cF$-free graphs or hypergraphs with large minimum degree have a property we call 
vertex-extendability (see Definition~\ref{DFN:vertex-extendable}).
The approach is designed for degree-stability and thus it not only simplifies the proofs 
of many known stability theorems but also gives stronger forms of these stability theorems.

\subsection{Main result}\label{SUBSEC:main-result}
Our results can be regarded as adding a new ingredient to Zykov's symmetrization 
method~\cite{Zy} and we commence by describing an `axiomatic' framework for the 
determination of extremal numbers by means of symmetrization. 

Given two $r$-graphs $F$ and $\cH$ a map $\phi\colon V(F) \to V(\cH)$ is said
to be a {\em homomorphism} if it preserves edges,
i.e., if $\phi(E) \in \cH$ holds for all $E \in F$.
If $\phi$ is surjective and every edge of~$\cH$ is an image of an edge of $F$, i.e., 
$\cH=\{\phi(E)\colon E\in F\}$, we call $\cH$ a \emph{homomorphic image} of $F$. 
Furthermore, $\cH$ is {\em $F$-hom-free} if there is no homomorphism from $F$ to $\cH$.
For a family $\cF$ of $r$-graphs,
we say that $\cH$ is {\em $\cF$-hom-free} if it is $F$-hom-free for every $F\in \cF$.
The forbidden families $\cF$ studied in this article have the following property.

\begin{definition}[Blowup-invariance]\label{DFN:blowup-invariant}
A family $\cF$ of $r$-graphs is \emph{blowup-invariant} if every $\cF$-free $r$-graph 
is $\cF$-hom-free as well.
\end{definition}

For instance, for every $\ell\ge 2$ the families of graphs $\{K_\ell\}$ and $\{C_3, \dots, C_{2\ell-1}\}$ are blowup invariant, whilst $\{C_5\}$ is not blowup-invariant. In the graph 
case one can easily check that a one-element family $\{F\}$ is blowup-invariant if and only 
if $F$ is a clique, but for hypergraphs blowup-invariant families consisting of a single 
hypergraph $F$ are much more common. In fact, if every pair of vertices of $F$ is covered 
by an edge of $F$, then $\{F\}$ is blowup-invariant. One confirms easily that every family $\cF$
closed under taking homomorphic images is blowup-invariant.

Let us now fix an $r$-graph $\cH$. For every $v\in V(\cH)$ we call 
\begin{align}
L_{\cH}(v) = \left\{A\in \binom{V(\cH)}{r-1}\colon A\cup\{v\}\in \cH\right\} \notag
\end{align}
the {\em link} of $v$. Two vertices $u,v\in V(\cH)$
are said to be {\em equivalent} if $L_{\cH}(u) = L_{\cH}(v)$.
Evidently, equivalence is an equivalence relation. 
We say that $\cH$ is {\em symmetrized}
if for any two non-equivalent vertices $u, v\in V(\cH)$ there is an edge $E\in\cH$ containing 
both of them. For instance, a symmetrized graph is the same as a complete multipartite graph.
We shall prove the following result by means of Zykov's symmetrization method.

\begin{theorem}\label{THM:Zykov-symmetrization-Turan-number}
Suppose that $\cF$ is a blowup-invariant family of $r$-graphs.
If $\gH$ denotes the class of all symmetrized $\cF$-free $r$-graphs,
then ${\rm ex}(n,\cF) = \gh(n)$ holds for every $n\in\NN$, where
$\gh(n) = \max\left\{|\cH|\colon \cH \in \gH \text{ and }v(\cH) = n\right\}$.
\end{theorem}

Let us observe that this statement is very similar to the Lagrangian method developed 
and utilised by Motzkin-Straus~\cite{MS65}, Sidorenko~\cite{Sido87}, 
Frankl-F\"{u}redi~\cite{FF89}, and many others. 
Preparing the statement of our main result, we introduce some further notions. 
Recall that a class $\gH$ of $r$-graphs is called \emph{hereditary} if it is closed under
taking induced subgraphs. 

\begin{definition}[Symmetrized-stability]\label{DFN:symmetrized-stable}
Let $\cF$ be a family of $r$-graphs and let $\gH$ be a class of $\cF$-free $r$-graphs. 
We say that $\cF$ is \emph{symmetrized-stable} with respect to $\gH$
if there exist $\epsilon>0$ and $N_0$ such that every
symmetrized $\cF$-free $r$-graphs $\cH$
on $n\ge N_0$ vertices with $|\cH| \ge (\pi(\cF)/r!-\epsilon)n^{r}$
is a subgraph of a member of $\gH$.
\end{definition}

The next definition might be the most important one in this article. 
 
\begin{definition}[Vertex-extendibility]\label{DFN:vertex-extendable}
Let $\cF$ be a family of $r$-graphs and let $\gH$ be a class of $\cF$-free $r$-graphs. 
We say that $\cF$ is \emph{vertex-extendable} 
with respect to $\gH$ if there exist $\zeta>0$ and $N_0\in\NN$ such that 
for every $\cF$-free $r$-graph $\cH$ on $n\ge N_0$
vertices satisfying $\delta(\cH)\ge \bigl(\pi(\cF)/(r-1)!-\zeta\bigr)n^{r-1}$ 
the following holds: if $\cH-v$ is a subgraph of a member of $\gH$ for some 
vertex $v\in V(\cH)$, then $\cH$ is a subgraph of a member of $\gH$ as well. 
\end{definition}

We can now state our sufficient conditions for degree-stability.

\begin{theorem}[Main result]\label{THM:Psi-trick:G-extendable-implies-degree-stability}
Suppose that $\cF$ is a blowup-invariant nondegenerate family of $r$-graphs and that 
$\gH$ is a hereditary class of $\cF$-free $r$-graphs. If~$\cF$ is symmetrized-stable 
and vertex-extendable with respect to~$\gH$, then~$\cF$ is degree-stable with respect 
to~$\gH$ as well.  
\end{theorem}

In practice the assumptions on $\gH$ are often easy to verify but it may happen that 
the family $\cF$ we want to study fails to be blowup-invariant. If in such a situation 
we know for any reason that $\cF$ is vertex-stable with respect to $\gH$, we can improve
this information to degree-stability. 

\begin{theorem}\label{PROP:vertex-stable-and-vertex-extendable-implies-degree-stability}
Suppose that $\cF$ is a nondegenerate family of $r$-graphs and that
$\gH$ is a hereditary class of $\cF$-free $r$-graphs.
If $\cF$ is vertex-stable and vertex-extendable with respect to $\gH$, then it 
is degree-stable with respect to $\gH$ as well. 
\end{theorem}

\subsection{Further results and applications}\label{subsec:0054}

An $r$-graph $\cH$ is said to be a \emph{blowup} of another $r$-graph $F$ is 
there exists a map $\psi\colon V(\cH)\lra V(F)$ such that every $E\in \binom{V(\cH)}r$
satisfies the equivalence $\psi(E)\in F \Longleftrightarrow E\in \cH$.  
If $\psi$ is surjective, the blowup is called \emph{proper}. 
Subgraphs of blowups of $F$ are called \emph{$F$-colorable}.  

For integers $\ell\ge r\ge 2$ let $\gK^r_\ell$ be the class 
of all blowups of $K^r_\ell$. 
If $r=2$ we omit the superscript and just write $\gK_\ell$ for the class of complete
$\ell$-partite graphs (whose vertex classes are allowed to be empty). Most but not all 
stability results described below are with respect to classes of the form $\gK^r_\ell$. 

\subsubsection{Graphs}

The classical stability theorem of Erd\H{o}s and Simonovits 
(Theorem~\ref{THM:Erdos-Simonovits-stability}) informs us that every family $\cF$ of graphs
with $\chi(\cF)=\ell+1\ge 3$ is edge-stable with respect to $\gK_\ell$. Complementing this 
result one can also characterise the 
families of graphs which are degree-stable and vertex-stable with respect to $\gK_\ell$. 
To this end we recall that a graph~$F$ is said to be \emph{edge-critical} if it has an 
edge $e\in F$ such that $\chi(F-e)<\chi(F)$ and \emph{matching-critical} if there exists 
an induced matching $M\subseteq F$ such that $\chi(F-M)<\chi(F)$. 
More generally, we call a family $\cF$ of graphs 
\emph{edge-critical} or \emph{matching-critical} if there exists a graph $F\in \cF$ 
with $\chi(F)=\chi(\cF)$ that is edge-critical or matching-critical.
In the result that follows, part~\ref{it:13b} is due to 
Erd\H{o}s and Simonovits~\cite{ES73}, while part~\ref{it:13a} might very well be new.

\begin{theorem}\label{thm:19}
	A family $\cF$ of graphs with $\chi(\cF)=\ell+1\ge 3$ is
	\begin{enumerate}[label=\alabel]
		\item\label{it:13a} vertex-stable with respect to $\gK_\ell$ if and only if it
			is matching-critical
		\item\label{it:13b} and degree-stable with respect to $\gK_\ell$ if and only if it
			is edge-critical.
	\end{enumerate}
\end{theorem}

\subsubsection{Cancellative hypergraphs and generalized triangles}

An $r$-graph $\cH$ is {\em cancellative} if $A\cup B = A\cup C$ implies that $B = C$ for 
all $A, B, C\in \cH$. Since $A\cup B = A\cup C$ is equivalent to $B \triangle C\subseteq A$,
an $r$-graph $\cH$ is cancellative if and only if it is $\cT_{r}$-free, where
$\cT_{r}$ denotes the family consisting of all $r$-graphs with three edges one of which 
contains the symmetric difference of the two other ones.

It was conjectured by Katona and proved by Bollob{\' a}s~\cite{BO74} that
the maximum number of edges in an $n$-vertex $\cT_3$-free $3$-graph is uniquely achieved by
the balanced complete $3$-partite $3$-graph. Keevash and the second author~\cite{KM04} 
proved that $\cT_3$ is edge-stable with respect to $\gK^3_3$, and 
the first author~\cite{LIU19} discovered another short proof of the edge-stability of~$\cT_3$
giving a linear dependency between the error parameters. Sidorenko~\cite{Sido87} proved that 
the maximum number of edges in an $n$-vertex $\cT_4$-free 
$4$-graph is uniquely achieved by the balanced complete $4$-partite $4$-graph.
Later, Pikhurko~\cite{PI08} proved that $\cT_4$ is vertex-stable with respect to $\gK^4_4$
using a sophisticated variation of  Zykov symmetrization. For $r\ge 5$ the value 
of~$\pi(\cT_{r})$ is unknown.

Cancellative hypergraphs are closely related to the Tur\'{a}n problem for generalized triangles. For $r\ge 2$ let $\Sigma_r$ be the collection of all $r$-graphs with three 
edges $A, B, C$ such that $|B\cap C| = r-1$ and $B\triangle C \subseteq A$. 
The unique $r$-graph $T_r\in \Sigma_r$ with $v(T_r)=2r-1$ is called the \emph{generalized 
triangle}. It is easy to see that $\Sigma_2 = \cT_2 = \{K_3\}$, $\Sigma_3 = \cT_3$, 
and $\Sigma_r\subsetneq \cT_r$ for~$r\ge 4$. 

The results on $\cT_4$ due to Sidorenko~\cite{Sido87} and Pikhurko~\cite{PI08} quoted 
earlier hold for $\Sigma_4$ instead of $\cT_4$ as well. In particular, $\Sigma_4$ is known to be 
vertex-stable with respect to $\gK^4_4$. For $r=5, 6$ Frankl and F\"{u}redi~\cite{FF89} 
proved that the extremal numbers $\ex(n, \Sigma_{r})$ are only realized by balanced blowups 
of the famous Witt designs~\cite{Witt} with parameters $(11,5,4)$ and $(12,6,5)$, respectively.
Norin and Yepremyan~\cite{NY17} proved that $\Sigma_{5}$ and $\Sigma_{6}$ are edge-stable
with respect to blowups of these Witt-designs, but Pikhurko showed~\cite{PI08} that they fail to 
be vertex-stable. For $r\ge 7$ it is an open problem to determine $\pi(\Sigma_r)$.

\begin{theorem}\label{THM:cancelltive-3-4-graphs-are-deg-stable}
	For $r\in\{3, 4\}$ the family $\Sigma_r$ is degree-stable with respect to $\gK^r_r$.
\end{theorem}

\subsubsection{Hypergraph expansions}
Given an $r$-graph $F$ and $i\in [r-1]$ we write $\partial_i F$ for the $(r-i)$-graph
with the same vertex set as $F$ whose edges are the $(r-i)$-subsets of~$V(F)$ covered 
by an edge of~$F$. In particular, $\partial_{r-2} F$ is a graph on $V(F)$. 
A set $X\subseteq V(F)$ is called {\it $2$-covered in $F$} if it induces a clique 
in $\partial_{r-2} F$. If $V(F)$ itself is $2$-covered in $F$ we simply say that $F$
is {\it $2$-covered}. The {\it neighborhood} $N_F(v)$ of a vertex $v\in V(F)$ is defined 
to be the set of all $u\in V(F)\sm\{v\}$ with $\{u, v\}\in \partial_{r-2} F$. 

For an $r$-graph $F$ with $\ell$ vertices we define $\cK^F_\ell$ to be the set 
of all $r$-graphs of the 
form $F\cup\bigl\{S_{uv}\colon uv\in \binom{V(F)}2\sm \partial_{r-2}F\bigr\}$,
where for every pair of vertices $uv\in \binom{V(F)}2\sm\partial_{r-2}F$ not covered 
by an edge of $F$ the edge $S_{uv}$ contains $u$ and $v$. We write $H^F_\ell$ for the 
unique member of $\cK^F_\ell$ having the largest number of vertices, namely 
\[
	v(H^{F}_{\ell}) = \ell + (r-2)\left(\binom{\ell}{2}-|\partial_{r-2}F|\right).
\]
The $r$-graphs in $\cK^F_\ell$ are called \emph{weak expansions} of~$F$ 
while~$H^F_\ell$ is called the \emph{expansion} of~$F$. If~$F$ has no edges (and thus 
consists of~$\ell$ isolated vertices) we write~$\cK_{\ell}^{r}$ and~$H_{\ell}^{r}$ instead 
of~$\cK^F_\ell$ and~$H^F_\ell$.

The notion of hypergraph expansions was first introduced by the second author in~\cite{MU06} 
to extend Tur\'{a}n's theorem to hypergraphs.
In~\cite{MU06} it was proved that for every $n \ge \ell \ge r\ge 2$ the maximum number of edges
in an $n$-vertex $\cK_{\ell+1}^{r}$-free $r$-graph is uniquely achieved by~$T_{r}(n,\ell)$,
the balanced complete $\ell$-partite $r$-graph on~$n$ vertices.
In addition,~\cite{MU06} proved that~$\cK_{\ell+1}^{r}$ is edge-stable with respect 
to $\{T_r(n, \ell)\colon n\in\NN\}$.
Later de Oliveira Contiero, Hoppen, Lefmann, and Odermann~\cite{deLLO19}, and independently, 
the first author~\cite{LIU19} improved the edge-stability result by showing that a linear 
dependence between $\delta$ and $\epsilon$ is sufficient.
Pikhurko~\cite{PI13} refined~\cite{MU06} by showing that~$T_{r}(n,\ell)$ is also the unique  
$H_{\ell+1}^{r}$-free $r$-graph on~$n$ vertices with the maximum number of edges for sufficiently large~$n$.

Keevash~\cite{KE11} observed a generalization of these results to expansions of a large class 
of hypergraphs $F$. Let us write $\lambda(\cG)$ for the Lagrangian of a hypergraph $\cG$ 
(see Section~\ref{SEC:preliminary} for the definition) and set 
$\pi_{\lambda}(F) = \sup\left\{\hbox{$\lambda(\cG)\colon \cG$ is $F$-free} \right\}$
for every $r$-graph $F$.

\begin{theorem}[Keevash]\label{THM:Turan-density-weak-expansion-F}
Let $F$ be an $r$-graph with $v(F) = \ell+1$. If $\pi_{\lambda}(F) \le \binom{\ell}{r}/\ell^r$,
then
\[
	{\rm ex}(n,\cK_{\ell+1}^{F}) \le {\binom{\ell}{r}} n^r/{\ell^r} 
\]
holds for all positive integers $n$, and equality occurs whenever $n$ is divisible by $r$.
In particular,
\[
	\pi(H_{\ell+1}^{F}) = \pi(\cK_{\ell+1}^{F}) = {(\ell)_{r}}/{\ell^r}.
\]
\end{theorem}

In the special case where $F$ has an isolated vertex and 
$\pi_{\lambda}(F) < \binom{\ell}{r}/\ell^r$ Brandt, Irwin, and Jiang~\cite{BIJ17}, and independently, Norin and Yepremyan~\cite{NY18}
proved a stability theorem for the family $\cK_{\ell+1}^{F}$ and used it to determine
${\rm ex}(n,H_{\ell}^{r})$ exactly for all sufficiently large integers $n$.
More specifically,~\cite{BIJ17} shows that
$\cK_{\ell+1}^{F}$ is vertex-stable, and~\cite{NY18}
shows that $\cK_{\ell+1}^{F}$ is edge-stable.
Our result below shows the stronger fact that $\cK_{\ell+1}^{F}$ is degree-stable.

Moreover, we prove degree-stability in many cases where $F$ has no isolated vertices
but is contained instead in the hypergraph $B(r, \ell+1)$ with vertex set $[\ell+1]$ 
and edge set 
\[
	\{[r]\}\cup\{E\subseteq [2, \ell+1]\colon |E|=r \text{ and } |[2, r]\cap E|\le 1\}.
\]

\begin{theorem}\label{thm:1837}
Let $\ell \ge r \ge 2$ and suppose that $F$ is an $r$-graph satisfying $v(F)=\ell+1$
and 
\begin{equation}\label{eq:1126}
	\sup\left\{\lambda(\cG)\colon 
	\text{$\cG$ is $F$-free and not $K^r_\ell$-colorable} \right\}< \binom{\ell}{r}/\ell^r.
\end{equation}
If either $F$ has an isolated vertex or $F\subseteq B(r, \ell+1)$, 
then the family $\cK_{\ell+1}^{F}$ is degree-stable with respect to $\gK^r_\ell$.
\end{theorem}

There are several natural examples of hypergraphs $F$ which have been proved to satisfy 
condition~\eqref{eq:1126} in the literature but whose families of weak 
expansions~$\cK^F_{\ell+1}$
were not known to be degree-stable before. For instance, Hefetz and Keevash~\cite{HK13} 
studied the case
that $F=M^3_2$ is a $3$-uniform matching with two edges and six vertices.  
More generally Jiang, Peng, and Wu~\cite{JPW18} proved the assumption 
if $F=\{M^3_t, L^3_t, L^4_t\}$ holds for some $t\ge 2$; here~$M^3_t$ denotes the $3$-uniform 
matching of size $t$ and for $r\ge 2$ the $r$-graph~$L^r_t$ consists of~$t$ edges having 
one vertex~$v$ in common such that $E\cap E'=\{v\}$ holds for all distinct $E, E'\in L^r_t$.
Brandt, Irwin, and Jiang~\cite{BIJ17} proved that in these cases the families $\cK^{F}$ are 
vertex-stable. By combining the results in~\cite{JPW18} on Lagrangians 
with Theorem~\ref{thm:1837} one immediately obtains the 
following strengthening of this fact. 

\begin{cor}
	For $t\ge 2$ the families $\cK^{M^3_t}_{3t}$, $\cK^{L^3_t}_{2t+1}$, $\cK^{L^4_t}_{3t+1}$ 
	are degree-stable with respect to $\gK^3_{3t-1}$, $\gK^3_{2t}$, $\gK^4_{3t}$, 
	respectively. \hfill $\Box$
\end{cor}

\subsubsection{Expansions of matchings of size 2}

For $r\ge 3$ let $M^r_2$ be the $r$-graph on $2r$ vertices consisting of two disjoint edges. 
The trivial observation that no $r$-graph in $\gK^r_{2r-1}$ contains a weak expansion 
of $M^r_2$ yields the lower bound $\pi(\cK^{M^r_2}_{2r})\ge (2r-1)_r/(2r-1)^r$. 
In their work~\cite{HK13} establishing equality for $r=3$ Hefetz and Keevash also observed 
that for $r\ge 4$ there is a denser construction of $\cK^{M^r_2}_{2r}$-free $r$-graphs. 

Call an $r$-graph $\cH$ \emph{semibipartite} if there exists a
partition $V(\cH)=A\dcup B$ such that $|A\cap E|=1$ holds for every $E\in \cH$. 
If $\cH$ contains all $|A|\binom{|B|}{r-1}$ such edges we say that $\cH$ is a 
\emph{complete semibipartite} $r$-graph. It is easy 
to see that semibipartite $r$-graphs cannot contain weak expansions of $M^r_2$ and that 
$(1-1/r)^{r-1}$ is the supremum of the edge densities of semibipartite $r$-graphs. 
A straightforward calculation shows that for $r\ge 4$ this number is indeed larger than the 
lower bound $(2r-1)_r/(2r-1)^r$ mentioned before. In fact Hefetz and Keevash~\cite{HK13} 
conjectured $\pi(\cK^{M^r_2}_{2r})=(1-1/r)^{r-1}$ for every $r\ge 4$. 
This was proved by Bene Watts, Norin, and Yepremyan~\cite{BNY19} 
who also showed that $\cK^{M^r_2}_{2r}$ is edge-stable with respect to the class 
$\gS^r$ of all complete semibipartite $r$-graphs. Combining a substantial result on 
Lagrangians from their work with our 
Theorem~\ref{THM:Psi-trick:G-extendable-implies-degree-stability} we strengthen this 
to degree-stability. 

\begin{theorem}\label{THM:degree-stable-of-hypergraph-expansion-two}
For every $r\ge 4$ the family of weak expansions of $M^r_2$ is degree-stable with respect 
to $\gS^r$.
\end{theorem}

We would like to point out that a different abstract framework for stability results based 
on Zykov's symmetrization method has recently been worked out by Liu, Pikhurko, Sharifzadeh, 
and Staden~\cite{LPSS}.

\subsection*{Organization}
In Section~\ref{SEC:preliminary} we introduce some definitions and useful lemmas.
The results presented in Subsection~\ref{SUBSEC:main-result} are proved in
Section~\ref{SEC:psi-trick-proof} and  Section~\ref{SEC:applications} deals with the 
applications described in Subsection~\ref{subsec:0054}. The last section consists
of concluding remarks. 

\section{Preliminaries}\label{SEC:preliminary}

We begin with some definitions related to the Lagrangians of hypergraphs (introduced
by Frankl-R\"odl~\cite{FR84}).
Given an $r$-graph $\cG$ on $m$ vertices
(let us assume for notational transparency that $V(\cG) = [m]$) the multilinear function
$L_{\cG}\colon \RR^m \to \RR$ is defined by
\begin{align}
L_{\cG}(x_1,\ldots,x_m) = \sum_{E\in \cH}\prod_{i\in E}x_i,
\quad \text{ for all } (x_1,\ldots,x_m) \in \RR^m. \notag
\end{align}
Denote by $\Delta_{m-1}$ the standard $(m-1)$-dimensional simplex,
i.e.
\begin{align}
\Delta_{m-1} = \left\{(x_1,\ldots,x_m)\in [0,1]^n\colon x_1+\cdots+x_m= 1\right\}. \notag
\end{align}
Since $\Delta_{m-1}$ is compact,
a theorem of Weierstra\ss\ implies that the restriction of $L_{\cG}$ to $\Delta_{m-1}$
attains a maximum value, called the {\em Lagrangian} of $\cG$ and denoted by $\lambda(\cG)$.

Lagrangians arise naturally when one considers the maximum possible densities 
of blowups.
Given an $r$-graph $\cG$ with vertex set $[m]$ and mutually disjoint sets 
$V_1, \dots, V_m$ we write $\cG[V_1, \dots, V_m]$ for the $r$-graph obtained
from $\cG$ upon replacing every vertex $i\in [m]$ by the set $V_i$ and every 
edge $\{i(1), \dots, i(r)\}\in\cG$ by the complete $r$-partite $r$-graph with 
vertex classes $V_{i(1)}, \dots, V_{i(r)}$. Obviously we 
have 
\begin{equation}\label{eq:1226}
	|\cG[V_1, \dots, V_m]|=L_\cG(|V_1|, \dots, |V_m|)
\end{equation} 
in this situation. 
If $\cH$ is a spanning subgraph 
of $\cG[V_1, \dots, V_m]$ we say that the partition $V(\cH)=\bigdcup_{i\in [m]}V_i$
is a \emph{$\cG$-coloring} of $\cH$. Observe that an $r$-graph $\cH$ admits
a $\cG$-coloring if and only if it is $\cG$-colorable. Since $L_\cG$ is a homogeneous 
polynomial of degree~$r$, the formula~\eqref{eq:1226} immediately implies the following 
observation (e.g. see \cites{FF89, KE11}).

\begin{lemma}\label{LEMMA:blowup-langrangian}
Let $\cG$ be an $r$-graph on $m$ vertices. If $\cH$ denotes an 
$r$-graph on $n$ vertices possessing a $\cG$-coloring $V(\cH) = \bigdcup_{i\in[m]}V_i$ 
and $x_i = |V_i|/n$ for every $i\in[m]$, then 
\[
\pushQED{\qed} 
|\cH| 
\le 
L_\cG(x_1,\ldots,x_m) n^r 
\le 
\lambda(\cG)n^r. \qedhere
\popQED
\]
\end{lemma}

The next result, concerning the Lagrangian of the complete $r$-graph $K_{m}^{r}$, 
will be useful in
proofs of stability theorems whose extremal configuration is a blowup of $K_{m}^{r}$.

\begin{lemma}\label{LEMMA:Lagrangian-complete-r-graph}
If $m \ge r \ge 2$ and $(x_1,\ldots,x_{m})\in \Delta_{m-1}$,
then
\begin{align}
L_{K_{m}^{r}}(x_1,\ldots,x_{m})
+ \frac{\binom{m}{r}}{m^{r-1}(m-1)}\sum_{i\in [m]}\left(x_i-\frac{1}{m}\right)^2
\le \frac{1}{m^{r}}\binom{m}{r}.\notag
\end{align}
\end{lemma}

Clearly this holds with equality if $(x_1, \dots, x_m)$ is either $(1/m, \dots, 1/m)$
or a unit vector. 

\begin{proof}[Proof of Lemma~\ref{LEMMA:Lagrangian-complete-r-graph}]
Setting $\mu_i = L_{K_{m}^{i}}(x_1,\ldots,x_m)$ for every $i\in [m]$,
we have $\mu_1=1$ 
and Maclaurin's inequality implies 
\begin{align}
\left(\frac{\mu_r}{\binom{m}{r}}\right)^{1/r}
\le \cdots \le
\left(\frac{\mu_2}{\binom{m}{2}}\right)^{1/2}
\le
\frac{\mu_1}{\binom{m}{1}}
=
\frac 1m. \notag
\end{align}
Consequently,
\begin{align}
\frac{\mu_2}{\binom{m}{2}}
\ge
\left(\frac{\mu_r}{\binom{m}{r}}\right)^{2/r}
\ge
\left(\frac{\mu_r}{\binom{m}{r}}\right)^{2/r} \cdot \left(m^r \frac{\mu_r}{\binom{m}{r}}\right)^{(r-2)/r},  \notag
\end{align}
i.e.,
\[
\mu_r \le \frac{\binom{m}{r}}{m^{r-2}}\cdot \frac{\mu_2}{\binom{m}{2}}.
\]
Since
\[
	2\mu_2 
	= 
	1 - \sum_{i\in[m]}x_i^2 
	= 
	\frac{m-1}{m} - \sum_{i\in[m]}\left(x_i-\frac{1}{m}\right)^2,
\]
the result follows.
\end{proof}

Easy calculations  based on Lemma~\ref{LEMMA:Lagrangian-complete-r-graph} 
imply the following.

\begin{corollary}\label{LEMMA:near-Turan-hypergrap-structure}
Given $m\ge r\ge 2$ let $\zeta>0$ be sufficiently small. Suppose further 
that~$\cH$ is an $n$-vertex $r$-graph admitting 
a $K^r_m$-coloring $V(\cH)=\bigdcup_{i\in[m]}V_i$.
\begin{enumerate}[label=\alabel]
\item\label{it:24a}
If $|\cH|\ge \left(\binom{m}{r}/m^r -\zeta\right)n^r$,
then $|V_i| = (1/m\pm C_1\zeta^{1/2}) n$ holds for every $i\in[m]$, 
where $C_1 = \left({m^{r-1}(m-1)}/{\binom{m}{r}}\right)^{1/2}$.
\item\label{it:24b}
If $\delta(\cH)\ge \bigl(\binom{m-1}{r-1}/m^{r-1}-\zeta\bigr)n^{r-1}$ and $v\in V_i$, 
then $|V_j\sm N_\cH(v)|\le 2C_1\zeta^{1/2} n$ holds for every $j\in [m]\sm\{i\}$.
\item\label{it:24c}
If $\delta(\cH)\ge \bigl(\binom{m-1}{r-1}/m^{r-1}-\zeta\bigr)n^{r-1}$,
then $\big|L_{\widehat{K}^r_m}(v)\sm L_\cH(v)\big|\le C_2\zeta^{1/2}n^{r-1}$
holds for every $v\in V(\cH)$, where $\widehat{K}^r_m=K_{m}^r[V_1, \ldots, V_m]$
and $C_2=r\binom{m-1}{r-1}C_1/m^{r-2}$. \qed
\end{enumerate}
\end{corollary}

The special case $r=3$ of the following lemma was stated in~\cite{LMR1}*{Lemma~4.5}. 
It is straightforward to generalize the probabilistic proof provided there to the general 
case and we omit the details.   

\begin{lemma}\label{LEMMA:greedily-embedding-Gi}
Fix a real $\eta \in (0, 1)$ and integers $m, n\ge 1$, $r\ge 3$.
Let $\cG$ be an $r$-graph with vertex set~$[m]$ and let $\cH$ be an $r$-graph
with $v(\cH)=n$.
Consider a vertex partition $V(\cH) = \bigdcup_{i\in[m]}V_i$ and the associated
blowup $\widehat{\cG} = \cG[V_1,\ldots,V_{m}]$ of $\cG$.
If two sets $T \subseteq [m]$ and $S\subseteq \bigcup_{j\not\in T}V_j$
have the properties
\begin{enumerate}[label=\alabel]
\item\label{it:25a} $|V_{j}| \ge (|S|+1)|T| \eta^{1/r} n$  for all $j \in T$,
\item\label{it:25b} $|\cH[V_{j_1},\ldots,V_{j_r}]| 
		\ge |\widehat{\cG}[V_{j_1},\ldots,V_{j_r}]| - \eta n^r$
      for all $\{j_1,\ldots,j_r\} \in \binom{T}{r}$, 
\item\label{it:25c} and $|L_{\cH}(v)[V_{j_1},\ldots,V_{j_{r-1}}]| 
		\ge |L_{\widehat{\cG}}(v)[V_{j_1},\ldots,V_{j_{r-1}}]| - \eta n^{r-1}$
      for all $v\in S$ and all $\{j_1,\ldots,j_{r-1}\} \in \binom{T}{r-1}$,
\end{enumerate}
then there exists a selection of vertices $u_j\in V_j$ for all $j\in [T]$
such that $U = \{u_j\colon j\in T\}$ satisfies
$\widehat{\cG}[U] \subseteq \cH[U]$ and
$L_{\widehat{\cG}}(v)[U] \subseteq L_{\cH}(v)[U]$ for all $v\in S$.
In particular, if $\cH \subseteq \widehat{\cG}$,
then $\widehat{\cG}[U] = \cH[U]$ and
$L_{\widehat{\cG}}(v)[U] = L_{\cH}(v)[U]$ for all $v\in S$. \qed

\end{lemma}

For the proof of the following standard fact we refer to~\cite{LMR1}*{Lemma~4.2}.

\begin{fact}\label{FACT:dense-hygp-has-large-min-deg}
Let $\cF$ be a family of $r$-graphs and let
$\cH$ be an $\cF$-free $r$-graph on $n$ vertices. If~$\cH$ has at least 
$\left(\pi(\cF)/r!-\epsilon\right)n^r$ edges, then
\begin{enumerate}[label=\alabel]
\item\label{it:26a}
the set 
$Z_{\epsilon}(\cH)= \left\{u\in V(\cH)\colon d_{\cH}(u) \le 
\left(\pi(\cF)/(r-1)!-2\epsilon^{1/2}\right)n^{r-1}\right\}$
has size at most~$\epsilon^{1/2}n$, 
\item\label{it:26b}
and the $r$-graph $\cH' = \cH-Z_{\epsilon}(\cH)$ satisfies
$\delta(\cH')> \left(\pi(\cF)/(r-1)!-3\epsilon^{1/2}\right)n^{r-1}$. \qed
\end{enumerate}
\end{fact}

\section{Proof of the main result: the \texorpdfstring{$\Psi$}{Psi}-trick}
\label{SEC:psi-trick-proof}

We prove Theorems~\ref{THM:Zykov-symmetrization-Turan-number},~\ref{THM:Psi-trick:G-extendable-implies-degree-stability}, 
and~\ref{PROP:vertex-stable-and-vertex-extendable-implies-degree-stability} in this section.
Let us recall that we call two vertices of a hypergraph equivalent if they have the same link. 
If $C$ denotes an 
equivalence class of some hypergraph $\cH$, we shall write $d_\cH(C)$ for the common degree of 
the vertices in~$C$ and~$L_\cH(C)$ for their common link. Given a class of hypergraphs $\gH$ 
we denote the class of spanning subgraphs of members of~$\gH$ by~$\gH^+$, i.e. we set 
\[
	\gH^+=\{\cH\colon \text{there is $\cH'\in \gH$ with $V(\cH)=V(\cH')$ 
		and $\cH\subseteq \cH'$} \}.
\]
If $\gH$ is hereditary, this is the same as the class of (not necessarily spanning) subgraphs
of members of $\gH$. 

\begin{proof}[Proof of Theorem~\ref{THM:Zykov-symmetrization-Turan-number}]
Fix $n\in \NN$. The lower bound ${\rm ex}(n,\cF) \ge \gh(n)$ is an immediate consequence of 
the fact that all members of $\gH$ are $\cF$-free. So it remains to establish 
the upper bound ${\rm ex}(n,\cF) \le \gh(n)$.

Suppose that it is not true and let $\cH$ be an $\cF$-free $r$-graph on $n$ vertices
with more than~$\gh(n)$ edges chosen in such a way that the number $m$ of its equivalence 
classes is minimum.
Let $C_1,\ldots, C_m$ denote the equivalence classes of $\cH$.

Due to $|\cH| > \gh(n)$ we know that $\cH$ cannot be symmetrized.
In other words, there exist $i,j\in [m]$ such that the graph $\partial_{r-2}\cH$ is not 
complete between $C_i$ and $C_j$.
Without loss of generality we may assume that $\{i,j\} = \{1,2\}$ 
and $d_{\cH}(C_1) \le d_{\cH}(C_2)$.
In view of the definition of equivalence there are actually no edges between $C_1$ and $C_2$ 
in $\partial_{r-2}\cH$.

Now let $\cH'$ be the unique $r$-graph satisfying $V(\cH') = V(\cH)$, $\cH'-C_1 = \cH-C_1$,
and $L_{\cH'}(v) = L_{\cH}(w)$ for all $v\in C_1$ and $w\in C_2$.
Observe that $\{C_1\cup C_2, C_3, \ldots, C_m\}$ is a refinement of the partition 
of $V(\cH')$ into equivalence classes of $\cH'$, for which reason $\cH'$ has fewer than~$m$
equivalence classes. Together with 
\begin{align}
|\cH'| = |\cH| + |C_1| \left(d_{\cH}(C_2) - d_{\cH}(C_1)\right) \ge |\cH| > \gh(n) \notag
\end{align}
and our minimal choice of $m$ this implies that $\cH'$ cannot be $\cF$-free.
As there exists a homomorphism from $\cH'$ to $\cH$,
it follows that $\cH$ fails to be $\cF$-hom-free. But, as $\cF$ is blowup-invariant, this 
contradicts the assumption that $\cH$ be $\cF$-free.
\end{proof}

For the proof of 
Theorem~\ref{THM:Psi-trick:G-extendable-implies-degree-stability}
it will be convenient to say for $\zeta>0$ and $N_0\in\NN$ that a family $\cF$
of $r$-graphs is \emph{$(\zeta, N_0)$}-vertex-extendable with respect to a class 
of $r$-graphs $\gH$ if using the notation of Definiton~\ref{DFN:vertex-extendable} 
$\zeta$ and $N_0$ exemplify the vertex-extendibility of $\cF$ with respect to $\gH$. 
The next lemma shows that vertex-extendibility can be used iteratively. 

\begin{lemma}\label{LEMMA:min-deg-imply-epsilon-G-extendable}
Suppose that the nondegenerate family $\cF$ of $r$-graphs is $(2\eps, N_0)$-vertex-extendable 
with respect to a class $\gH$ of $\cF$-free $r$-graphs, where $\eps\in (0, 1/2)$ 
and $N_0\in \NN$.
Let~$\cH$ be an $\cF$-free $r$-graph on $n\ge 2N_0$ vertices 
satisfying $\delta(\cH) \ge \bigl(\pi(\cF)/(r-1)!-\epsilon\bigr)n^{r-1}$. 
If there exists a set $S\subseteq V(\cH)$ with $|S|\le \eps n$ and $(\cH-S)\in \gH^+$,
then $\cH\in \gH^+$.
\end{lemma}

\begin{proof}[Proof of Lemma~\ref{LEMMA:min-deg-imply-epsilon-G-extendable}]
Choose a minimal set $S'\subseteq S$ with $(\cH-S')\in \gH^+$. 
If $S'=\varnothing$ we are done, so suppose for the sake of contradiction that 
there exists a vertex $v\in S'$. Setting $S''=S'\sm\{v\}$ and $\cH''=\cH-S''$ we have 
$v(\cH'')\ge n-|S|\ge (1-\eps)n\ge n/2\ge N_0$ and 
\begin{align*}
	\delta(\cH'')
	&\ge 
	\delta(\cH) - |S''| n^{r-2}
 	> 
	\bigl(\pi(\cF)/(r-1)!-\epsilon\bigr)n^{r-1} - \epsilon n^{r-1} \\
 	&\ge
	\bigl(\pi(\cF)/(r-1)!-2\epsilon\bigr)v(\cH'')^{r-1}. 
\end{align*}
Moreover, we are assuming that $\cH''-v=\cH-S'$ is in $\gH^+$. 
So by vertex-extendibility  $\cH''$ belongs to $\gH^+$ as well and $S''$ contradicts the 
minimality of $S'$.
\end{proof}

Next we shall show the following strengthening of 
Theorem~\ref{THM:Psi-trick:G-extendable-implies-degree-stability} which also allows 
vertices of low degree in the almost extremal $\cF$-free graphs. Recall that the sets
$Z_\eps(\cH)$ appearing below were defined in 
Fact~\ref{FACT:dense-hygp-has-large-min-deg}~\ref{it:25a}.

\begin{theorem}\label{THM:Psi-trick:G-extendable-implies-degree-stability-full-version}
Let $\cF$ be a blowup-invariant nondegenerate family of $r$-graphs 
and let $\gH$ be a hereditary class of $\cF$-free $r$-graphs. 
If $\cF$ is symmetrized-stable and vertex-extendable with respect to $\gH$, then there 
are $\eps>0$ and $N_0\in \NN$ such that every 
$\cF$-free $r$-graph $\cH$ on $n\ge N_0$ vertices with
$|\cH| > \left(\pi(\cF)/r!-\epsilon\right)n^{r}$
satisfies $\cH-Z_{\epsilon}(\cH) \in \gH^+$.
\end{theorem}

The proof involves the following invariant of hypergraphs: If $C_1, \ldots, C_m$
are the equivalence classes of an $r$-graph $\cH$, we set $\Psi(\cH)=\sum_i|C_i|^2$.

\begin{proof}[Proof of Theorem~\ref{THM:Psi-trick:G-extendable-implies-degree-stability-full-version}]
Choose $\eps\in (0, 1/36)$ so small and $N_0\in\NN$ so large that 
\begin{enumerate}[label=\nlabel]
	\item\label{it:psi-a} the symmetrized stability of $\cF$ with respect to $\gH$
		is exemplified by $\eps$ and $N_0$
	\item\label{it:psi-b} and $\cF$ is $(6\eps^{1/2}, N_0/3)$-vertex-extendable with respect 
		to $\gH$.
\end{enumerate}

Now we fix $n\ge N_0$ and, assuming that the conclusion fails for some $n$-vertex 
hypergraph~$\cH$,
we pick a counterexample $\cH$ with $v(\cH)=n$ such that the pair $(|\cH|, \Psi(\cH))$ is 
lexicographically maximal (which makes sense, as $n$ is fixed).
Let $C_1, \ldots, C_m$ be the equivalence classes of~$\cH$.

Setting $Z = Z_{\epsilon}(\cH)$ we have $(\cH-Z) \not\in \gH^+$ and,
as $\gH$ is hereditary, $\cH\not\in\gH^+$ follows. Now~\ref{it:psi-a} informs us 
that $\cH$ is not symmetrized. 
So without loss of generality 
we may suppose that $\partial_{r-2} \cH$ has no edges between $C_1$ and $C_2$
and that $(d_{\cH}(C_1), |C_1|)\le_{\mathrm{lex}} (d_{\cH}(C_2), |C_2|)$,
where $\le_{{\rm lex}}$ means lexicographic ordering.

Now we pick two arbitrary vertices $v_1\in C_1$ and $v_2\in C_2$
and symmetrize only them, 
i.e., we let $\cH'$ be the $r$-graph with $V(\cH')=V(\cH)$,
$\cH'-v_1 = \cH-v_1$ and $L_{\cH'}(v_1)=L_{\cH}(v_2)$.
Clearly, if $d_{\cH}(v_1)<d_{\cH}(v_2)$, then $|\cH'|>|\cH|$.
Moreover, if $d_{\cH}(v_1)=d_{\cH}(v_2)$, then $|\cH'|=|\cH|$, $|C_1|\le |C_2|$, and
\begin{align}
\Psi(\cH')-\Psi(\cH) \ge (|C_1|-1)^2+(|C_2|+1)^2-|C_1|^2-|C_2|^2=2(|C_2|-|C_1|+1)\ge 2. \notag
\end{align}
In both cases $(|\cH'|, \Psi(\cH'))$ is lexicographically larger than $(|\cH|, \Psi(\cH))$
and our choice of $\cH$ implies $\cH'-Z_{\epsilon}(\cH') \in \gH^+$.
By Fact~\ref{FACT:dense-hygp-has-large-min-deg}~\ref{it:26a}
the set $Q=Z_{\epsilon}(\cH')\cup\{v_1\}$ 
satisfies $|Q|\le \epsilon^{1/2} n+1 < 2\epsilon^{1/2} n$.
Since $\gH$ is hereditary, the $r$-graph $\cH-Q=\cH'-Q$ belongs to~$\gH^+$.  
Now Fact~\ref{FACT:dense-hygp-has-large-min-deg}~\ref{it:26b} and~\ref{it:psi-b}
show that Lemma~\ref{LEMMA:min-deg-imply-epsilon-G-extendable} applies to 
$3\eps^{1/2}$, $2N_0/3$, $\cH-(Q\cap Z)$, and $Q\sm Z$ here in place of $\eps$, $N_0$, $\cH$, 
and $S$ there. Thus $\cH-(Q\cap Z)$ belongs to $\gH^+$ and, since~$\gH$ is hereditary, this 
yields the contradiction $(\cH-Z)\in\gH^+$.
\end{proof}

\begin{proof}[Proof of Theorem~\ref{THM:Psi-trick:G-extendable-implies-degree-stability}]
If $\cH$ satisfies $\delta(\cH) > \left(\pi(\cF)/(r-1)!-\epsilon\right)n^{r-1}$,
then
\[
	Z_{\epsilon}(\cH) = \emptyset 
	\quad \text{ and } \quad
	|\cH| = \frac{1}{r}\sum_{v\in V(\cH)}d_{\cH}(v) > \left(\pi(\cF)/r!-\epsilon\right)n^{r}.
\]
Therefore, Theorem~\ref{THM:Psi-trick:G-extendable-implies-degree-stability-full-version} implies
Theorem~\ref{THM:Psi-trick:G-extendable-implies-degree-stability}.
\end{proof}

Finally we prove Theorem~\ref{PROP:vertex-stable-and-vertex-extendable-implies-degree-stability}
using Lemma~\ref{LEMMA:min-deg-imply-epsilon-G-extendable}.

\begin{proof}[Proof of 
Theorem~\ref{PROP:vertex-stable-and-vertex-extendable-implies-degree-stability}]
Pick $\delta\in (0, 1/2)$ and $N'_0\in\NN$ such that $\cF$ 
is $(2\delta, N'_0)$-vertex-extendable 
with respect to $\gH$. The vertex stability of $\cF$ with respect to $\gH$ applied 
to~$\delta$ yields 
some~$\eps>0$ and $N_0\in \NN$ (see Definition~\ref{d:1715}~\ref{it:stab-v}). We may assume 
that $\eps\le \delta$ and $N_0\ge 2N'_0$.

Now let $\cH$ be an $\cF$-free $r$-graph on $n\ge N_0$ vertices with
$\delta(\cH) \ge \bigl(\pi(\cF)/(r-1)!-\epsilon\bigr)n^{r-1}$. We are to prove $\cH\in\gH^+$.
It follows from $r|\cH| = \sum_{v\in V(\cH)}d_{\cH}(v)$
that $|\cH| \ge \left(\pi(\cF)/r!-\epsilon\right)n^{r}$.
So by the vertex-stability of $\cF$ there exists a set $B\subseteq V(\cH)$ of size at 
most $\delta n$
such that the $r$-graph $\cH' = \cH-B$ is a member of $\gH^+$.
Since $\delta(\cH) \ge \left(\pi(\cF)/(r-1)!-\delta\right)n^{r-1}$ and $n\ge 2N'_0$
it follows from Lemma~\ref{LEMMA:min-deg-imply-epsilon-G-extendable} 
and our choice of $\delta$ that $\cH\in\gH^+$.
\end{proof}
\section{Applications}\label{SEC:applications}

\subsection{Graphs}\label{SUBSEC:complete-graphs}
In this subsection we prove Theorem~\ref{thm:19}. As we have already mentioned, 
its part~\ref{it:13b} is due to Erd\H{o}s and Simonovits~\cite{ES73}, 
who proved that for every 
edge-critical graph~$F$ with $\chi(F)=\ell+1\ge 3$ there exists some $N_0\in\NN$ such that 
every graph $G$ on $n\ge N_0$ vertices whose minimum degree is larger 
than $\frac{3\ell-4}{3\ell-1}n+O(1)$
either contains $F$ or is $\ell$-colorable. Consequently, all edge-critical 
families of graphs are degree-stable with respect to $\gK_\ell$.  

Now suppose, conversely, that some graph family $\cF$ is degree-stable with respect to the 
class $\gK_\ell$, where $\ell\ge 2$. This means, in particular, that for every $n\ge \ell+1$ 
the graph~$T^+(n, \ell)$
obtained from the $n$-vertex $\ell$-partite Tur\'an graph by inserting an additional edge 
into one of its vertex classes cannot be $\cF$-free. Moreover, there cannot exist a 
graph $F'\in\cF$ with $\chi(F')\le \ell$, for then some member of $\gK_\ell$ would fail to 
be $\cF$-free (as demanded by Definition~\ref{d:1715}).
So altogether, $\cF$ needs to contain
an edge-critical graph $F$ with $\chi(F)=\chi(\cF)=\ell+1$. 
In other words, $\cF$ is indeed edge-critical.

We are left with proving part~\ref{it:13a} of Theorem~\ref{thm:19}. The forward implication
from vertex stability to matching-criticality is very similar to the argument in the previous
paragraph, but instead of the graphs $T^+(n, \ell)$ one considers graphs $T^M(n, \ell)$ 
obtained from 
Tur\'an graphs by inserting (almost) perfect matchings into one of their partition classes.
Clearly, this matching is induced in $T^M(n, \ell)$.  
Omitting further details we proceed to the backwards implication. It clearly suffices to treat families consisting of a single graph. 

\begin{lemma}\label{lem:1748}
	Let $F$ be a graph with $\chi(F)=\ell+1\ge 3$. If $F$ is matching-critical, then $F$ 
	is vertex-stable with respect to $\gK_\ell$.
\end{lemma}

\begin{proof}[Proof of Lemma~\ref{lem:1748}]
Given $\delta>0$ we choose $\eps, \eta>0$ and $N_0\in \NN$ obeying the 
hierarchy $N_0^{-1}\ll\eps\ll\eta\ll\delta$. 
Suppose that $G$ is an $F$-free graph on $n\ge N_0$ vertices with at least
$\left(\frac{\ell-1}{2\ell}-\epsilon\right)n^2$ edges. We are to prove that 
$G$ can be made $\ell$-partite by deleting at most $\delta n$ vertices. 
Theorem~\ref{THM:Erdos-Simonovits-stability} yields a partition 
$V(G)=\bigdcup_{i\in [\ell]}V_i$ such that $\sum_{i\in [\ell]}|G[V_i]|\le \eta n^2$.
Set 
\[
	X_i=\left\{x\in V_i\colon \text{ there is $j\in[\ell]\sm\{i\}$ such 
		that $|V_j\sm N(x)|\ge \frac n{3\ell v(F)}$}\right\}
\]
for every $i\in [\ell]$. Since $|G[V_i, V_j]|\ge |V_i||V_j|-2\eta n^2$ holds for all 
distinct $i, j\in [\ell]$, we have $|X_i|\le 6(\ell-1)\ell v(F) \eta n \le \delta n/2\ell$ 
for every $i\in [\ell]$.

Recall that there is an induced matching $M$ such that $\chi(F-M)\le \ell$. 
If for some $i\in [\ell]$
there are $|M|$ independent edges $e_1, \dots, e_{|M|}$ in $G[V_i\sm X_i]$ we can find a copy 
of $F$ in $G$ where these edges $e_1, \dots, e_{|M|}$ play the r\^{o}le of $M$. 
So by $F\not\subseteq G$ such matchings do not exist and it follows  
that for every $i\in [\ell]$ there is a set $Y_i\subseteq V_i\sm X_i$ of size $|Y|\le 2|M|$ 
covering all edges. 
Now the set $Q=\bigcup_{i\in [\ell]}(X_i\cup Y_i)$ has size at most  
$\delta n/2+2\ell |M|\le \delta n$ and $G-Q$ is $\ell$-partite.   
\end{proof}


\subsection{Cancellative hypergraphs and generalized triangles}\label{SUBSEC:cancellative-hypergraphs}

The goal of this subsection is to deduce 
Theorem~\ref{THM:cancelltive-3-4-graphs-are-deg-stable} from 
Theorem~\ref{THM:Psi-trick:G-extendable-implies-degree-stability}.
We commence by introducing a class~$\gT_r$ of $\Sigma_r$-free $r$-graphs
which is larger than $\gK^r_r$.

For integers $n\ge r \ge \ell \ge 1$ we call an $r$-graph $\cG$ on $n$ vertices 
an \emph{$(n,r,\ell)$-system} if every $\ell$-subset of $V(\cG)$ is contained in 
at most one edge. As shown in~\cites{Sido87, FF89, PI08, LM1, Liu20a}, the Tur\'an problem 
for $\Sigma_r$ is closely related to $(n,r,r-1)$-systems. 
Given any $r\ge 3$ we write~$\gT_r$ for the class of all blowups 
of $2$-covered $(n,r,r-1)$-systems. Since $K^r_r$
is a $2$-covered $(r,r,r-1)$-system, we have $\gK^r_r\subseteq \gT_r$. Perhaps at first sight 
surprisingly, we shall apply Theorem~\ref{THM:Psi-trick:G-extendable-implies-degree-stability}
to $\cF=\Sigma_r$ and $\gH=\gT_r$. This choice of $\gH$ is forced upon us due to the 
symmetrized stability assumption and the following fact. 

\begin{lemma}\label{LEMMA:cancelltive-3-4-graphs-symmetrized-stable}
For $r\ge 3$ a $\Sigma_r$-free $r$-graph is symmetrised if and only if it 
is a proper blowup of some $2$-covered $(n,r,r-1)$-system.
\end{lemma}

\begin{proof}[Proof of Lemma~\ref{LEMMA:cancelltive-3-4-graphs-symmetrized-stable}]
Suppose first that $\cH$ is a symmetrised $\Sigma_r$-free $r$-graph. Being a symmetrised
hypergraph, $\cH$ is a proper blowup of some $2$-covered $r$-graph $\cT$. If $\cT$ fails to be 
a $(v(\cT), r, r-1)$-system, then there are edges $B, C\in \cT$ such that $|B\cap C|=r-1$.
Since~$\cT$ is 2-covered, some edge $A\in\cT$ contains the two-element set $B\triangle C$.
Now $\{A, B, C\}$ is a subgraph of $\cT$ belonging to $\Sigma_r$, contrary to $\cH$ 
being $\Sigma_r$-free. This proves that $\cT$ 
is indeed a $(v(\cT), r, r-1)$-system. 

In the converse direction, proper blowups of $2$-covered $(n,r,r-1)$-systems
are clearly symmetrised and an argument similar to the previous paragraph shows 
that they are $\Sigma_r$-free as well.
\end{proof}

Proceeding with our intended application 
of Theorem~\ref{THM:Psi-trick:G-extendable-implies-degree-stability}
we observe that due to being closed under the formation of homomorphic images $\Sigma_r$
is blow-up invariant. Moreover, $\gT_r$ is clearly hereditary and the previous lemma 
shows that $\Sigma_r$ is symmetrized-stable with respect to $\gT_r$. So it remains to verify 
vertex-extendibility for $r\in \{3,4\}$. As the following lemma demonstrates, 
for this task we may restrict our attention to $\gK^r_r$ rather than $\gT_r$. 

\begin{lemma}\label{LEMMA:cancelltive-3-4-graphs-symmetrized-stable-high-degree}
For $r\in\{3, 4\}$ there exists $\eps_r>0$ such that every $\cH\in \gT_r$ with 
minimum degree $\delta(\cH)>(r^{1-r}-\eps_r)n^{r-1}$ belongs to $\gK^r_r$. 
\end{lemma}

\begin{proof}[Proof of Lemma~\ref{LEMMA:cancelltive-3-4-graphs-symmetrized-stable-high-degree}]
Choose $\eps_3, \eps_4>0$ sufficiently small and
suppose that for some ${r\in \{3, 4\}}$ an $r$-graph $\cH \in \gT_r$ has $n$ vertices and 
minimum degree at least $\left(r^{1-r}-\epsilon_r\right)n^{r-1}$.
Without loss of generality we can suppose that~$\cH$ is a proper blowup of
some (not necessarily $2$-covered) $(m, r, r-1)$-system $\cT$ with $V(\cT) = [m]$. 
Write $\cH=\cT[V_1, \dots, V_m]$ and set $x_i = |V_i|/n$ for every~$i\in [m]$.

Since $d_{\cH}(v) = L_{L_{\cT}(i)}(x_1,\ldots,x_m) n^{r-1}$ holds for all $v\in V_i$ and $i\in [m]$,
the minimum degree assumption yields 
$L_{L_{\cT}(i)}(x_1,\ldots,x_m) \ge r^{1-r} - \epsilon_r$ for every $i \in [m]$.
On the other hand, as every $(r-1)$-subset of $V(\cT)$ in contained in at most one 
edge of $\cT$, we have 
$\sum_{i\in [m]}L_{L_{\cT}(i)}(x_1,\ldots,x_m) \le  L_{K^{r-1}_m}(x_1, \dots, x_m)$. 
It follows that
\begin{equation}\label{eq:1130}
\left(r^{1-r}-\epsilon_r\right)m
\le \sum_{i\in [m]}L_{L_{\cT}(i)}(x_1,\ldots,x_m)
\le \lambda(K^{r-1}_m)
= \binom{m}{r-1}/m^{r-1}.
\end{equation}

Now for $r=3$ a sufficiently small choice of $\eps_3$ 
guarantees $m \in \{2,3\}$; so $\cT$ consists of a single edge and $\cH \in \gK_{3}^{3}$.
In the $4$-uniform case~\eqref{eq:1130} leads to $m\in\{4, 5\}$; since there exists
no $2$-covered $(5, 4, 3)$-system, the case $m=5$ is impossible and thus we have 
indeed~$\cH\in\gK^4_4$. 
\end{proof}

Due to the lower bound $\pi(\Sigma_r)\ge r!/r^r$, which follows from the 
fact that $r$-graphs in~$\gK^r_r$ are $\Sigma_r$-free, the next lemma will 
imply that for $r\in\{3, 4\}$ the family $\Sigma_r$ is vertex-extendable 
with respect to $\gT_r$. 

\begin{lemma}\label{LEMMA:cancellative-3-4-graphs-vertex-extendable}
For every integer $r \ge 2$ there exist $\zeta>0$ and $N_0\in\NN$ such that 
every $\Sigma_r$-free $r$-graph $\cH$ on $n\ge N_0$ vertices which has minimum degree 
$\delta(\cH)\ge (r^{1-r}-\zeta)n^{r-1}$ and possesses a vertex $v$ such $\cH-v$
is $K^r_r$-colorable is $K^r_r$-colorable itself. 
\end{lemma}

\begin{proof}[Proof of Lemma~\ref{LEMMA:cancellative-3-4-graphs-vertex-extendable}]
Given $r\ge 2$ we choose appropriate constants $\zeta>0$ and $N_0\in\NN$ fitting into 
the hierarchy 
$N_0^{-1}\ll \zeta\ll r^{-1}$. Now let $\cH$ be a $\Sigma_{r}$-free $r$-graph
on $n\ge N_0$ vertices whose minimum degree is at least $(r^{1-r}-\zeta)n^{r-1}$.
Set $V = V(\cH)$ and suppose that some vertex $v\in V$ has the property that 
$\cH_v = \cH-v$ is $K_{r}^{r}$-colorable. 
Fix a $K_{r}^{r}$-coloring $V(\cH_v) = \bigdcup_{i\in[r]}V_i$ of $\cH_v$.
Clearly
\begin{equation}\label{eq:1724}
\delta(\cH_v) 
\ge 
\left(r^{1-r}-\zeta\right)n^{r-1} - n^{r-2} 
\ge 
\left(r^{1-r}-2\zeta\right)n^{r-1}
\end{equation}
and Corollary~\ref{LEMMA:near-Turan-hypergrap-structure}~\ref{it:24a} yields
\begin{equation}\label{eq:1714}
	|V_i| = \left(1/r\pm \zeta^{1/3}\right)n 
	\quad \text{ for all $i\in [r]$}.
\end{equation}

\begin{claim}\label{CLAIM:link-v-is-empty-inside-Vi-cancellative}
Every edge of $\cH$ intersects every vertex class $V_i$ in at most one vertex. 
\end{claim}

\begin{proof}[Proof of Claim~\ref{CLAIM:link-v-is-empty-inside-Vi-cancellative}]
By symmetry it suffices to show $|E\cap V_1| \le 1$ for every $E\in \cH$.
Assume for the sake of contradiction that there exist distinct vertices 
$w_1,w_1' \in E\cap V_1$. The $(r-1)$-graphs $G_1 = L_{\cH_{v}}(w_1)$ and 
$G_1' = L_{\cH_{v}}(w_1')$ are $(r-1)$-partite with vertex partition $V_2\dcup \cdots \dcup V_r$ 
and by~\eqref{eq:1724} both of them have at least the size $(1/r^{r-1}-2\zeta)n^{r-1}$.
Due to~\eqref{eq:1714} this implies $|G_1\cap G_1'| \ge n^{r-1}/2r^{r-1}$
and, in particular, there exists an edge $e\in G_1\cap G_1'$. 
Now $\left\{E, e\cup\{w_1\}, e\cup \{w_1'\}\right\} \in \Sigma_{r}$ contradicts the 
assumption that $\cH$ is $\Sigma_{r}$-free.
\end{proof}

Since no edge of $L_{\cH}(v)$ can intersect all the partition classes $V_1, \dots, V_r$
we may assume without loss of generality that at least $d(v)/r$ edges 
of $L_{\cH}(v)$ are contained in $V_2\dcup\dots\dcup V_r$.

\begin{claim}\label{CLAIM:v-has-no-neighbor-in-V1-cancellative}
We have $N_{\cH}(v) \cap V_1 = \emptyset$.
\end{claim}

\begin{proof}[Proof of Claim~\ref{CLAIM:v-has-no-neighbor-in-V1-cancellative}]
Suppose to the contrary that there exists a vertex $u\in N_{\cH}(v) \cap V_1$
and consider an edge $E \in \cH$ containing $\{u, v\}$.
Let $G_u$ and $G_v$ be the subgraphs of~$L_{\cH}(u)$ and~$L_{\cH}(v)$
induced by $\bigcup_{j\in[2,r]}V_j$ respectively. Clearly, $G_u$ is $(r-1)$-partite 
and by Claim~\ref{CLAIM:link-v-is-empty-inside-Vi-cancellative}~$G_v$ is $(r-1)$-partite 
as well. Moreover,~\eqref{eq:1724} yields $|G_u| \ge \left(1/r^{r-1}-2\zeta\right)n^{r-1}$.
Together with $|G_v| \ge d(v)/r\ge\left(1/r^{r-1}-\zeta\right)n^{r-1}/r$ and~\eqref{eq:1714}
this implies
\[
	|G_u\cap G_v| \ge \frac{1}{2r} \frac{n^{r-1}}{r^{r-1}}.
\]
But if $e \in G_u\cap G_v$ is arbitrary, then the subgraph $\{E, e\cup\{v\}, e\cup\{u\}\}$ 
of $\cH$ belongs to~$\Sigma_{r}$, contrary to $\cH$ being $\Sigma_r$-free.
\end{proof}

By Claim~\ref{CLAIM:link-v-is-empty-inside-Vi-cancellative}
and Claim~\ref{CLAIM:v-has-no-neighbor-in-V1-cancellative}
the partition $V(\cH)=\bigdcup_{i\in[r]}\widehat{V}_i$ defined by
\begin{align}
\widehat{V}_i =
\begin{cases}
V_1\cup\{v\} & \mbox{ if } i =1, \\
V_i & \mbox{ if } i \in[2,r],
\end{cases} \notag
\end{align}
is a $K_{r}^{r}$-coloring of $\cH$.
This completes the proof of Lemma~\ref{LEMMA:cancellative-3-4-graphs-vertex-extendable}.
\end{proof}

We have thereby checked all assumptions of 
Theorem~\ref{THM:Psi-trick:G-extendable-implies-degree-stability}
and can conclude that for $r\in\{3, 4\}$ the family $\Sigma_r$ is 
degree-stable with respect to $\gT_r$. 
In view of Lemma~\ref{LEMMA:cancelltive-3-4-graphs-symmetrized-stable-high-degree}
this implies that $\Sigma_r$ is degree-stable with respect to $\gK^r_r$ as well. 

\subsection{Hypergraph expansions}\label{SUBSEC:hypergraph-expansions-part-one}

Throughout this subsection we fix two integers $\ell\ge r\ge 2$ and an $r$-graph~$F$
with $\ell+1$ vertices satisfying the assumptions
of Theorem~\ref{thm:1837}. Our goal is to conclude 
from Theorem~\ref{THM:Psi-trick:G-extendable-implies-degree-stability} that the 
family $\cK^F_{\ell+1}$ is indeed degree-stable with respect to $\gK^r_\ell$.

Since the family $\cK^F_{\ell+1}$ is closed under taking homomorphic images, it is 
blowup-invariant 
and, clearly, $\gK^r_\ell$ is hereditary. So it remains to show that $\cK^F_{\ell+1}$ is
symmetrized-stable and vertex-extendable with respect to $\gK^r_\ell$. The fact that all 
members of $\gK^r_\ell$ are $\cK^F_{\ell+1}$-free implies 
$\pi(\cK^F_{\ell+1})\ge (\ell)_r/\ell^r$ and thus our claim on symmetrized stability follows 
from the next statement.

\begin{lemma}\label{LEMMA:weak-expansion-family-is-sym-stable}
	There exists some $\eps>0$ such that every symmetrized 
	$\cK^F_{\ell+1}$-free $r$-graph~$\cH$ with $n$ vertices 
	and $|\cH|>\left(\binom{\ell}{r}/\ell^{r}-\eps\right)n^{r}$
	is $K^r_\ell$-colorable.
\end{lemma}

\begin{proof}
	We contend that every positive number $\eps$ satisfing 
	\[
		\sup\left\{\lambda(\cG)\colon 
			\text{$\cG$ is $F$-free but not $K^r_\ell$-colorable} \right\}+\eps 
		\le 
		\binom{\ell}{r}/\ell^r
	\]
	has the desired property. 	To see this we consider an arbitrary symmetrized 
	$\cK^F_{\ell+1}$-free $r$-graph $\cH$ with $n$ vertices 
	and $|\cH|>\left(\binom{\ell}{r}/\ell^{r}-\eps\right)n^{r}$.
	Since $\cH$ is symmetrized, there exists a $2$-covered hypergraph $\cG$ such that 
	$\cH$ is a proper blow-up of $\cG$. Now $|\cH| \le \lambda(\cG)n^r$
	yields $\binom{\ell}r/\ell^r-\eps<\lambda(\cG)$.
	On the other hand, since~$\cH$ is $\cK^F_{\ell+1}$-free and~$\cG$ is $2$-covered, 
	$\cG$ must be $F$-free. So our choice of $\eps$ implies that $\cG$ 
	is $K^r_\ell$-colorable and, hence, so is $\cH$.
\end{proof}

The next lemma implies that $\cK^F_{\ell+1}$ is vertex-extendable with respect to $\gK^r_\ell$
and thus concludes the proof of Theorem~\ref{thm:1837}.

\begin{lemma}\label{LEMMA:weak-expansions-F-is-vtx-extendable}
There exist $\zeta>0$ and $N_0\in \NN$ such that every $\cK^F_{\ell+1}$-free $r$-graph $\cH$ 
on $n\ge N_0$ vertices satisfying the minimum degree condition 
$\delta(\cH)>\left(\binom{\ell-1}{r-1}/\ell^{r-1}-\zeta\right)n^{r-1}$ and 
possessing a vertex $v$ such that $\cH-v$ is $K^r_\ell$-colorable is $K^r_\ell$-colorable
itself.
\end{lemma}

A slight modification of the proof below
shows that this holds for $H_{\ell+1}^{F}$ instead of the family $\cK_{\ell+1}^{F}$ 
as well. 

\begin{proof}[Proof of Lemma~\ref{LEMMA:weak-expansions-F-is-vtx-extendable}]
Choose $N_0^{-1}\ll\zeta\ll\ell^{-1}$ appropriately and let $\cH$ be a $\cK_{\ell+1}^{F}$-free $r$-graph on $n\ge N_0$ vertices whose minimum degree is at least 
$\left(\binom{\ell-1}{r-1}/\ell^{r-1}-\zeta\right)n^{r-1}$.
Write $V = V(\cH)$ and suppose that  
$\cH_{v} = \cH-v$ is $K_{\ell}^{r}$-colorable for some vertex $v\in V$.
Consider a {$K_{\ell}^{r}$-coloring} $\bigdcup_{i\in[\ell]} V_{i} = V\setminus\{v\}$ of $\cH_v$
and the associated blowup 
$\widehat{K}_{\ell}^{r} = K_{\ell}^{r}[V_1,\ldots,V_{\ell}]$ of~$K_{\ell}^{r}$.
Sets in $\widehat{K}_{\ell}^{r} \setminus \cH_{v}$ are called \emph{missing edges} of $\cH_v$; 
furthermore, for every $u\in V$ sets in $L_{\widehat{K}_{\ell}^{r}}(u)\setminus L_{\cH_v}(u)$ are 
called \emph{missing edges} of $L_{\cH_v}(u)$.

Notice that
\begin{align}\label{equ:delta-Hv}
\delta(\cH_v) 
\ge \delta(\cH) - n^{r-2} 
\ge \left(\binom{\ell-1}{r-1}/\ell^{r-1}-2\zeta\right)n^{r-1}.
\end{align}
Due to $|\cH|\ge n\delta(\cH)/r>\left(\binom{\ell}{r}/\ell^{r}-\zeta\right)n^{r}$
we have, similarly, 
\begin{align}\label{equ:Hv-size}
|\cH_v|
\ge
|\cH|-n^{r-1}  
\ge 
\left(\binom{\ell}{r}/\ell^r-2\zeta\right)n^{r}.
\end{align}
Consequently, the number of missing edges of $\cH_{v}$ is at most $2\zeta n^r$.
We proceed with a weak version of Corollary~\ref{LEMMA:near-Turan-hypergrap-structure}.

\begin{claim}\label{CLAIM:routine-results-expansion-one}
The following hold.
\begin{enumerate}[label=\alabel]
\item\label{it:413a}
We have $|V_i| = \left(1/\ell \pm \zeta^{1/3}\right)n$ for every  $i\in [\ell]$.
\item\label{it:413b}
If $i\in [\ell]$ and $u\in V(\cH_v)\sm V_i$, then $|V_i\sm N_\cH(u)|\le \zeta^{1/3}n$.
\item\label{it:413c}
For every $u\in V(\cH_v)$ the number of missing edges of $L_{\cH_{v}}(u)$ is at 
most $\zeta^{1/3} n^{r-1}$. \qed
\end{enumerate}
\end{claim}

The proof of our next claim exploits the fact that $F$ fails to be $2$-covered, i.e., 
that there exist two distinct vertices $u, v\in V(F)$ such that $uv\not\in\partial_{r-2} F$.
Indeed, if $F$ has an isolated vertex this is clear and if $F\subseteq B(r, \ell+1)$ 
we can take $u=1$ as well as $v=r+1$. 

\begin{claim}\label{CLAIM:link-v-is-empty-inside-Vi}
We have $|E\cap V_i|\le 1$ for all $E\in\cH$ and $i\in[\ell]$.
\end{claim}

\begin{proof}[Proof of Claim~\ref{CLAIM:link-v-is-empty-inside-Vi}]
Otherwise we may assume, without loss of generality, that for some edge $E$ 
there exist two distinct vertices $w_1,w_1' \in E\cap V_1$.
By Claim~\ref{CLAIM:routine-results-expansion-one}~\ref{it:413a} and~\ref{it:413b} for every 
$i\in [2,\ell]$ the set $V_i' = V_i\cap N_{\cH}(w_1) \cap N_{\cH}(w'_1)$
satisfies $|V_i'| > n/2\ell$.
Applying Lemma~\ref{LEMMA:greedily-embedding-Gi} with $S = \{w_1,w'_1\}$ and $T=[2,\ell]$
we obtain vertices $u_i \in V_i'$ for $i\in[2,\ell]$ such that the set 
$U = \{u_i\colon i\in[2,\ell]\}$ satisfies 
$\cH[U\cup \{w_1\}] \cong \cH[U\cup \{w'_1\}] \cong K_{\ell}^{r}$.
As $F$ is not $2$-covered, it is a subgraph of $H = \cH[U\cup \{w_1, w'_1\}]$.
Thus $H\cup \{E\}$ is a weak expansion of $F$, contrary to $\cH$ being $\cK^F_{\ell+1}$-free. 
\end{proof}

Essentially it remains to be shown that $N_{\cH}(v) \cap V_i=\varnothing$ holds for some 
$i\in [\ell]$. Preparing ourselves we first show the following weaker result.  

\begin{claim}\label{CLAIM:neighbor-v-is-empty-in-some-Vi}
There is no index $i\in [\ell]$ such that $N_{\cH}(v)\cap V_i \ne \emptyset$
and $|N_{\cH}(v) \cap V_j| \ge 2\zeta^{1/4r} n$ for all $j\in [\ell]\sm\{i\}$.
\end{claim}

\begin{proof}[Proof of Claim~\ref{CLAIM:neighbor-v-is-empty-in-some-Vi}]
By symmetry it suffices to deal with the case $i=1$. Assume for the sake of 
contradiction that there exists a vertex $u_1\in N_{\cH}(v)\cap V_1$ and moreover, 
that $|N_{\cH}(v) \cap V_j| \ge 2\zeta^{1/4r} n$ for all $j\in [2, \ell]$.
We shall show that, contrary to the hypothesis, $\cH$ contains a weak expansion of $F$.

Suppose first that $F$ has an isolated vertex. Due to
Claim~\ref{CLAIM:routine-results-expansion-one}~\ref{it:413b} for every $j\in[2,\ell]$
the set $V_j' = V_j\cap N_{\cH}(v) \cap N_{\cH}(u_1)$ 
has at least the size $|V_j'| \ge 2 \zeta^{1/4r} n - \zeta^{1/3} n > \zeta^{1/4r} n$.
So we can apply Lemma~\ref{LEMMA:greedily-embedding-Gi} to $S = \{u_1\}$ and $T=[2,\ell]$, 
thus obtaining a set $U = \{u_j\colon j\in[2,\ell]\}$ with $u_j \in V_j'$ for $j\in[2,\ell]$ 
and $\cH[U\cup \{u_1\}] \cong K_{\ell}^{r}$.
For every $i\in [\ell]$ let $e_i \in \cH$ be an edge containing both $u_i$ and $v$.
Since at least one vertex of $F$ is isolated, 
$H = \cH[U\cup \{u_1\}] \cup \{e_j\colon j\in[\ell]\}$
is the desired weak $(\ell+1)$-expansion of $F$.

So it remains to consider the case $F\subseteq B(r, \ell+1)$.
Pick an edge $E\in \cH$ containing $\{v,u_1\}$.
By Claim~\ref{CLAIM:link-v-is-empty-inside-Vi} we may assume that $E$ is of the form 
$\{v, u_1, \ldots, u_{r-1}\}$, where $u_j\in V_j$ holds for all $j\in [2,r-1]$.
Claim~\ref{CLAIM:routine-results-expansion-one}~\ref{it:413b} tells us that for every 
$k\in[r,\ell]$ the set 
\[
	V_k' = V_k\cap N_{\cH}(v) \cap \left(\bigcap_{j\in[1,r-1]}N_{\cH}(u_j)\right)
\]
has at least the size $|V_k'| \ge 2 \zeta^{1/4r} n - (r-1)\zeta^{1/3} n > \zeta^{1/4r} n$. 
For this reason Lemma~\ref{LEMMA:greedily-embedding-Gi} 
applied to $S = \{u_1, \ldots,u_{r-1}\}$
and $T=[r,\ell]$ leads to a set $U = \{u_k\colon k\in[r,\ell]\}$ such that 
\begin{enumerate}
	\item[$\bullet$] $u_k \in V_k'$ for every $k\in[r,\ell]$, 
	\item[$\bullet$] $\cH[U] \cong K_{\ell-r+1}^{r}$,
	\item[$\bullet$] and $L_{\cH}(u_j)[U] \cong K_{\ell-r+1}^{r-1}$ for every $j\in[r-1]$.
\end{enumerate}
Next we select for every $k \in [r, \ell]$ an edge $E_i \in \cH$ containing both $u_k$ and $v$. 
Now
\[
	H = \cH[U\cup \{u_1,\ldots, u_{r-1}\}] \cup \{E\} \cup \{E_k\colon k\in[r,\ell]\}
\]
is a weak expansion of $B(r, \ell+1)$ and, a fortiori, a weak expansion of $F$.
\end{proof}

Let us now consider the set 
\[
	S 
	= 
	\left\{i \in [\ell] \colon |N_{\cH}(v) \cap V_i| \ge 2\zeta^{1/4r} n\right\}. 
\]
By Claim~\ref{CLAIM:neighbor-v-is-empty-in-some-Vi} we know, in particular, that $S\ne [\ell]$.
Pick an arbitrary $i_\star\in [\ell]\sm S$. 
Now Claim~\ref{CLAIM:routine-results-expansion-one}~\ref{it:413a} 
and $|d_{\cH}(v)| \ge \left(\binom{\ell-1}{r-1}/\ell^{r-1}-\zeta\right)n^{r-1}$
imply $S = [\ell]\sm\{i_\star\}$ and a further application of 
Claim~\ref{CLAIM:neighbor-v-is-empty-in-some-Vi} 
discloses $N_{\cH}(v) \cap V_{i_\star}=\varnothing$. Together with 
Claim~\ref{CLAIM:link-v-is-empty-inside-Vi} this shows that 
the partition $V(\cH)=\bigdcup_{i\in[\ell]}\widehat{V}_i$ defined by
\begin{align}
\widehat{V}_i =
\begin{cases}
V_{i_\star}\cup \{v\} & \mbox{ if } i =i_\star, \\
V_i & \mbox{ if } i\ne i_\star,
\end{cases} \notag
\end{align}
is a $K_{\ell}^{r}$-coloring of $\cH$.
This completes the proof of Lemma~\ref{LEMMA:weak-expansions-F-is-vtx-extendable}.
\end{proof}

\subsection{Expansions of Matchings of size \texorpdfstring{$2$}{2}}
\label{SUBSEC:hypergraph-expansions-part-two}

In this subsection we shall derive Theorem~\ref{THM:degree-stable-of-hypergraph-expansion-two}
from Theorem~\ref{THM:Psi-trick:G-extendable-implies-degree-stability}. Again it is easy 
to see that $\cK^{M^r_2}_{2r}$ is blowup-invariant and that the class~$\gS^r$ is hereditary. 
Bene Watts, Norin, and Yepremyan proved in~\cite{BNY19} that 
\[
\sup\left\{\lambda(\cG)\colon 
	\text{$\cG$ is $M^r_2$-free but not semibipartite}  \right\}
	< 
	\frac{(1-1/r)^{r-1}}{r!}
\]
holds for all $r\ge 4$, where, let us recall, the numerator is the supremum of the edge 
densities of semibipartite $r$-graphs. Following the proof of 
Lemma~\ref{LEMMA:weak-expansion-family-is-sym-stable} one easily deduces from this result 
that $\cK^{M^r_2}_{2r}$ is symmetrized-stable with respect to $\gS^r$. So it only remains 
to establish vertex-extendibility, i.e., the following lemma.  

\begin{lemma}\label{LEMMA:2-matching-vertex-extendable-star}
Let $r \ge 4$ and $F = M_{2}^{r}$.
There exist $\zeta>0$ and $N_0\in\NN$ such that every $\cK^{F}_{2r}$-free $r$-graph $\cH$ 
on $n\ge N_0$ vertices satisfying 
$\delta(\cH) \ge \left(\left(1-\frac{1}{r}\right)^{r-1}/(r-1)!-\zeta\right)n^{r-1}$
and possessing a vertex $v$ for which $\cH-v$ is semibipartite
is semibipartite itself. 
\end{lemma}

In order to estimate the sizes of the vertex classes of semibipartite hypergraphs
with almost the maximum number of edges we use the following estimate. 

\begin{fact}\label{LEMMA:analysis-star-lagrangian}
If $r\ge 2$ and $x\in[0,1]$, then
\begin{align}
\frac{x(1-x)^{r-1}}{(r-1)!} + \frac{1}{r!}\left(1-\frac{1}{r}\right)^{r-3}\left(x-\frac{1}{r}\right)^2
\le  \frac{1}{r!}\left(1-\frac{1}{r}\right)^{r-1}. \notag
\end{align}
\end{fact}

Note that equality holds for $x=1/r$ and $x=1$.

\begin{proof}[Proof of Fact~\ref{LEMMA:analysis-star-lagrangian}]
	The case $x=1$ being clear we assume $x\in [0, 1)$ from now on.
	The standard inductive proof of Bernoulli's inequality also shows $(1+2h)(1+h)^{r-2}\ge 1+rh$
	for every real $h\ge -1$. In particular, for $h=(x-1/r)/(1-x)$ we obtain 
	\[
		 \left(\frac{1-1/r}{1-x}\right)^{r-2}\frac{1+x-2/r}{1-x}\ge \frac{(r-1)x}{1-x}\,.
	\]
	Multiplying by $(1-x)^{r}$ we deduce
	\begin{align*}
		(r-1)x(1-x)^{r-1}
		&\le 
		(1-1/r)^{r-2}(1-x)(1+x-2/r) \\
		&=
		(1-1/r)^{r-2}[(1-1/r)^2-(x-1/r)^2]
	\end{align*}
	and now it remains to divide by $(r-1)(r-1)!$.
\end{proof}

\begin{proof}[Proof of Lemma~\ref{LEMMA:2-matching-vertex-extendable-star}]
Fix some sufficiently small $\zeta\ll r^{-1}$ 
and then some sufficiently large $N_0\gg\zeta^{-1}$. 
Let $\cH$ be a $\cK^{F}_{2r}$-free $r$-graph on $n\ge N_0$ vertices 
whose minimum degree is at least
$\left(\left(1-\frac{1}{r}\right)^{r-1}/(r-1)!-\zeta\right)n^{r-1}$.
Set $V = V(\cH)$ and suppose that for some vertex~$v\in V$ the $r$-graph
$\cH_v = \cH-v$ is semibipartite.
Fix a partition $V(\cH_v) = V_1\cup V_2$ such that $|E\cap V_1| =1$ holds for every 
$E\in \cH_{v}$
and let $\widehat{\cS}$ be the complete semibipartite $r$-graph on~$V(\cH_v)$ 
corresponding to this partition.
Sets in $\widehat{\cS}\setminus \cH_v$ are called \emph{missing edges of $\cH_v$},
and for every $u\in V\setminus \{v\}$ sets in 
$L_{\widehat{\cS}}(u)\setminus L_{\cH_{v}}(u)$ are called \emph{missing edges of 
$L_{\cH_{v}}(u)$}.

As usual we have
\begin{align}
\delta(\cH_v) \ge \left(\left(1-\frac{1}{r}\right)^{r-1}/(r-1)!-2\zeta\right)n^{r-1}
\quad{\rm and}\quad
|\cH_v| \ge \left(\left(1-\frac{1}{r}\right)^{r-1}/r!-2\zeta\right)n^{r}. \notag
\end{align}
In particular, the number of missing edges of $\cH_{v}$ is at most $2\zeta n^r$.

\begin{claim}\label{CLAIM:routine-results-expansion-two-star}
The following statements hold.
\begin{enumerate}[label=\alabel]
\item\label{it:414a}
We have $|V_1| = \left(1/r \pm \zeta^{1/3}\right)n$ 
and $|V_2| = \left((r-1)/r \pm \zeta^{1/3}\right)n$.
\item\label{it:414b}
For every $u\in V(\cH_v)$ the number of missing edges of $L_{\cH_{v}}(u)$ 
is at most $\zeta^{1/3} n^{r-1}$.
\item\label{it:414c}
If $u\in V_1$, then $|V_2\sm N_{\cH_{v}}(u)| \le \zeta^{1/3} n$.
\item\label{it:414d}
If $u\in V_2$, then $|N_{\cH_{v}}(u)| \ge \left(1 - \zeta^{1/3}\right)n$.
\end{enumerate}
\end{claim}

\begin{proof}
Setting $x=|V_1|/n$ we have 
\[
	2\zeta
	>
	\frac{(1-1/r)^{r-1}}{r!}-\frac{|\widehat{\cS}|}{n^r}
	>
	\frac{(1-1/r)^{r-1}}{r!}-\frac{x(1-x)^{r-1}}{(r-1)!}
\]
and due to $\zeta\ll r^{-1}$ Fact~\ref{LEMMA:analysis-star-lagrangian} 
leads to $|x-1/r|\le O_r(\zeta^{1/2})\le \zeta^{1/3}$, which proves~\ref{it:414a}.
Moroever, in $\widehat{\cS}$ every vertex has 
degree $\left(\left(1-\frac{1}{r}\right)^{r-1}/(r-1)!\pm O_r(\zeta^{1/2})\right)n^{r-1}$
and thus for every $u\in V(\cH_v)$ there are at most $O_r(\zeta^{1/2}) n^{r-1}$
missing edges of $L_{\cH_{v}}(u)$, which implies~\ref{it:414b}. 
Now for part~\ref{it:414c} it suffices to observe that every vertex 
in $|V_2\sm N_{\cH_{v}}(u)|$ belongs to $\Omega_r(n^{r-2})$ missing edges of $L_{\cH_{v}}(u)$
and the argument for~\ref{it:414d} is similar. 
\end{proof}

Since $\cH$ contains no weak expansion of $M^r_2$, there cannot exist two disjoint edges 
$E, E'\in \cH$ such that $E\cup E'$ is $2$-covered. 

\begin{claim}\label{clm:1253}
	If two distinct vertices $u, w\in V(\cH)$ 
	satisfy 
	\[
		|L_\cH(u)[V_2]|, |L_\cH(w)[V_2]|
		\ge 
		\left(\left(1-\frac{1}{r}\right)^{r-1}/(r-1)!-\zeta^{1/4}\right)n^{r-1},
	\] 
	then no edge of $\cH$ contains both of them.  
\end{claim}

\begin{proof}[Proof of Claim~\ref{CLAIM:link-v-is-empty-inside-Vi-expansion-two-star}]
Assume contrariwise that some edge $E\in \cH$ contains $u$ and $w$. 
We shall show that this leads to two disjoint edges $E_u$, $E_w$
of $\cH$ such that $u\in E_u\subseteq V_2\cup\{u\}$, $w\in E_{w}\subseteq V_2\cup\{w\}$, and $E_{u}\cup E_{w}$ is $2$-covered,
which is absurd.

Owing to Claim~\ref{CLAIM:routine-results-expansion-two-star}~\ref{it:414a}
and our assumption on the links of $u$ and $w$ we have 
\[
	|V_2\sm N_{\cH}(u)|,  |V_2\sm N_{\cH}(w)| \le \zeta^{1/5} n.
\]
The latter estimate and our lower bound on $|L_\cH(u)[V_2]|$ show that there exists 
an edge $E_{u}\in \cH$ such that $u\in E_{u}$ and $E_u\sm\{u\}\subseteq V_2\cap N_{\cH}(w)$.
Now Claim~\ref{CLAIM:routine-results-expansion-two-star}~\ref{it:414d} 
and our upper bound on $|V_2\sm N_{\cH}(u)|$
imply that the set $V_2' = \bigcap_{x\in E_u} N_{\cH}(x) \cap (V_2\sm E_u)$
has at least the size $|V_2'| \ge |V_2|- 2\zeta^{1/5} n$.
Thus there exists an edge $E_{w} \in \cH_v$ with $w\in E_{w}\subseteq V_2'\cup\{w\}$.
Clearly $E_u$ and $E_{w}$ are as desired. 
\end{proof}

By our lower bound on $\delta(\cH_v)$ any two distinct vertices $u, w\in V_1$ satisfy 
the hypothesis of Claim~\ref{clm:1253}, which has the following consequence.

\begin{claim}\label{CLAIM:link-v-is-empty-inside-Vi-expansion-two-star}
	We have $|E\cap V_1| \le 1$ for every $E\in \cH$. \qed 
\end{claim}

Notice that $d_{\cH}(v) \ge \left(\left(1-\frac{1}{r}\right)^{r-1}/(r-1)!-\zeta\right)n^{r-1}$
yields 
\[	
	|N_{\cH}(v)| \ge \left(1-1/r -O_r(\zeta)\right)n\ge 2n/3, 
\]
whence
\begin{align}\label{eq:0126}
	|N_{\cH}(v)\cap V_2| \ge 2n/3-|V_1|\ge n/3.
\end{align}
 
If there exists no edge $E_\star\in \cH$ with $v\in E_\star\subseteq V_2\cup\{v\}$,
then $V(\cH)=V_1\dcup (V_2\cup \{v\})$ is a partition  exemplifying that $\cH$ is semibipartite
and we are done. So we may suppose from now on that such an edge $E_\star$ exists. 
Consider the set $X= \bigcap_{w\in E_\star}N_{\cH}(w)$.
On the one hand, Claim~\ref{CLAIM:routine-results-expansion-two-star}~\ref{it:414d} 
and~\eqref{eq:0126} imply
\[
	|X \cap V_2| 
	\ge 
	|N_\cH(v) \cap V_2| - (r-1) \zeta^{1/3} n 
	\ge 
	n/3 - n/12 
	= 
	n/4.
\]
On the other hand, there cannot exist an edge $E'\subseteq X\sm E_\star$, 
for then $\{E_\star, E'\}$ would be a matching in $\cH$ such that  
$E_\star\cup E'$ is $2$-covered.  
Since there are at most $2\zeta n^r$ missing edges, this 
implies $|X\cap V_1|\le O_r(\zeta) n\le \zeta^{1/3}n$.
As Claim~\ref{CLAIM:routine-results-expansion-two-star}~\ref{it:414d} 
yields $|N_\cH(v)\setminus X|\le (r-1)\zeta^{1/3}n$, we may conclude
\[
	|N_\cH(v)\cap V_1|\le |N_\cH(v)\setminus X|+|X\cap V_1|
	\le 
	r\zeta^{1/3}n,
\]
whence 
\[
	|L_\cH(v)[V_2]|
	\ge 
	d_\cH(v)-|N_\cH(v)\cap V_1|n^{r-2}
	\ge
	\left(\left(1-1/r\right)^{r-1}/(r-1)!-\zeta^{1/4}\right)n^{r-1}.
\]
Now Claim~\ref{clm:1253} 
discloses $N_{\cH}(v) \cap V_1 = \emptyset$. In view of 
Claim~\ref{CLAIM:link-v-is-empty-inside-Vi-expansion-two-star} this shows that the
partition $V(\cH)=(V_1\cup\{v\})\cup V_2$ witnesses the semibipartiteness of $\cH$.
\end{proof}

\section{Concluding remarks}\label{SEC:remarks}
$\bullet$ In this article we provided a framework for proving the degree-stability of 
certain classes of graph and hypergraph families, and applied it to the degree-stability 
of $\Sigma_{3}$, 
$\Sigma_4$, and $\cK_{\ell}^{F}$ for some combinations of $F$ and $\ell$.
In fact, one could push our results further and show that $T_{3}$, $T_{4}$, and $H_{\ell}^{F}$ (for some combinations of $F$ and $\ell$) 
are degree-stable by using the degree-stability results obtained here, applying the 
Removal lemma to prove the vertex-stability of $T_{3}$, $T_{4}$, and $H_{\ell}^{F}$, respectively, and finally applying 
Theorem~\ref{PROP:vertex-stable-and-vertex-extendable-implies-degree-stability}.

$\bullet$ Generalizing Theorem~\ref{thm:19} one may attempt to characterize
for arbitrary $\ell\ge r\ge 2$ the hypergraph families which are vertex-stable or 
degree-stable with respect to $\gK^r_\ell$. This problem is presumably very difficult 
and even partial results in this direction would be interesting. 

$\bullet$ A classical example in hypergraph Tur\'{a}n theory suggested by Vera T.~S\'os 
is the Fano plane, i.e. 
the $3$-graph on vertex set $[7]$ with edge set 
\begin{align}
\{123, 345, 561, 174, 275, 376, 246\}. \notag
\end{align}
The Tur\'{a}n density of the Fano plane was determined by De Caen and F\"{u}redi in~\cite{DF00}. 
Later Keevash and Sudakov~\cite{KS05b} and, independently, F\"{u}redi and Simonovits~\cite{FS05}
proved the degree-stability of the Fano plane and used it to determine the Tur\'{a}n number 
for large $n$.
The complete determination of its Tur\'{a}n number was obtained only recently by Bellmann and 
the third author~\cite{BR19}.
We do not know whether our method can be used to give another proof of the degree-stability 
of the Fano plane.

$\bullet$ Recall that by Theorem~\ref{THM:Erdos-Simonovits-stability} every family $\cF$ of 
graphs with $\chi(\cF)=\ell+1$ is edge-stable with respect to the 
family $\{T(n, \ell)\colon n\in \NN\}$ of Tur\'an graphs, which has the property that for every 
$n\in\NN$ it contains a {\it unique} $n$-vertex graph. This state of affairs prompted the 
second author~\cite{MU07} to define for every nondegenerate family $\cF$ of $r$-graphs the 
(edge-) stability number $\xi_e(\cF)$ to be the least number $t$ such that there exists a class 
of $r$-graphs~$\gH$ with the following properties:
\begin{enumerate}
	\item[$\bullet$] $\cF$ is edge-stable with respect to $\gH$;
	\item[$\bullet$] for every $n\in\NN$ there are $t$ hypergraphs on $n$ vertices in $\gH$.  
\end{enumerate}

For instance, the families studied in this article have stability number $1$ and standard 
conjectures imply that the stability number of $K^3_4$ is infinite. It was shown 
recently~\cites{LM2, LMR1} that for every $t\in\NN$ there exists a family $\cM_t$ 
of triple systems such that $\xi_e(\cM_t)=t$. 

In analogy with Definition~\ref{d:1715} one can also define a \emph{vertex-stability 
number} $\xi_v(\cF)$ and a \emph{degree-stability number} $\xi_d(\cF)$. These satisfy  
the easy estimates $\xi_{e}(\cF) \le \xi_{v}(\cF) \le \xi_{d}(\cF)$ and it would be interesting 
to study how ``exotic'' these parameters can get. 

$\bullet$ Our method can also be used in the context of other combinatorial structures, 
such as families of edge-weigthed graphs. To give an example, we recall the following 
result of Erd\H{o}s, Hajnal, S\'os, and Szemer\'edi~\cite{EHSS} 
from \emph{Ramsey-Tur\'an theory}: For $r\ge 2$ every $K_{2r}$-free graph 
with $n$ vertices and more than  $\bigl(\frac{3r-5}{3r-2}+o(1)\bigr)n^2/2$ edges contains 
an independent set of size $o(n)$. Here the constant $\frac{3r-5}{3r-2}$ is optimal and the
analogous problem with forbidden cliques of odd order is much easier. The proof of this result 
involves a certain family $\cF_{2r}$ of graphs with weights from $\{0, 1/2, 1\}$ assigned to 
their edges. The main points of the argument are $(i)$ that $\pi(\cF_{2r})=\frac{3r-5}{3r-2}$ and 
$(ii)$ that the regularity method establishes a connection between~$\cF_{2r}$ and $K_{2r}$.
L\"uders and the third author~\cite{LR-a} recently obtained the sharper result that 
for $\delta\ll r^{-1}$
every $K_{2r}$-free graph with $n$ vertices and more 
than $\bigl(\frac{3r-5}{3r-2}+\delta-\delta^2\bigr)n^2/2$ edges contains an independent set of 
size $\delta n$, where the term $\frac{3r-5}{3r-2}+\delta-\delta^2$ is again optimal. 
Their proof requires some stability result for the family $\cF_{2r}$. In fact, they provide
a rather ad-hoc proof of vertex-stability (see~\cite{LR-a}*{Proposition 5.5}) and returned to 
the topic in~\cite{LR-b} proving degree-stability. A straightforward adaptation of the 
$\Psi$-trick to weighted graphs yields an alternative (and shorter) proof of the 
degree-stability of $\cF_{2r}$. 

$\bullet$ We would like to emphasize that the strongest general stability result  
in this article, Theorem~\ref{THM:Psi-trick:G-extendable-implies-degree-stability-full-version},
can also be used for giving reasonable quantitative versions of edge stability. 
For instance, combined with the results in Subsection~\ref{SUBSEC:cancellative-hypergraphs}
it tells us that if $\eps>0$ is sufficiently small, then every $\Sigma_4$-free quadruple 
system $\cH$ on a sufficiently large number $n$ of vertices with more than $(1/256-\eps)n^4$
edges admits a partition $V(\cH)=A\dcup B\dcup C\dcup D\dcup Z$ such that $|Z|\le \eps^{1/2}n$
and $\cH-Z$ is $4$-partite with vertex classes $A$, $B$, $C$, and $D$. Moreover, all vertices 
in $V(\cH)\sm Z$ have at least the degree $(1/64-2\eps^{1/2})n^3$. 
Now a careful calculation 
shows $|A|, |B|, |C|, |D|=(1/4\pm 6\eps^{1/2})n$ and the proof of 
Claim~\ref{CLAIM:link-v-is-empty-inside-Vi-cancellative} discloses that the 
sets $A$, $B$, $C$, and~$D$ are independent in $\partial_2 \cH$. By the proof of  
Claim~\ref{CLAIM:v-has-no-neighbor-in-V1-cancellative}, if some $z\in Z$ satisfies 
$|L_\cH(z)[A, B, C]|\ge 4\eps^{1/2}n^3$, then $z$ has no neighbours in $D$. 
So $\cH$ can be made $4$-partite by the deletion of at most $17\eps n^4$ edges, namely
$(i)$ at most $\eps n^4$ edges with two or more vertices in $Z$;
$(ii)$ at most $4\eps n^4$ edges $zabc$ with $z\in Z$, $a\in A$, $b\in B$, $c\in C$, and
$|L_\cH(z)[A, B, C]|\ge 4\eps^{1/2}n^3$; $(iii)$ and,
similarly, at most $4\eps n^4$ edges of each of the three 
			types $zabd$, $zacd$, $zbcd$.
In particular, the edge stability of $\Sigma_4$ with respect to~$\gK^4_4$ holds with a linear 
dependence between the error terms. Taking into account that at most $400\eps n$ vertices 
$v\in V(\cH)$ can satisfy $d_{\cH}(v)\le n^3/80$ one can show the stronger result 
that~$\cH$ can be made $K^4_4$-colorable by the deletion of $7000\eps^{3/2}n^4$ edges,
which seems to be a new result. 

\section*{Acknoledgement}
We would like to thank both referees for reading this article very carefully
and making valuable suggestions.


\end{document}